%% file: EHMMIIIarXivv1.tex
\documentclass[12pt, oneside]{amsart}
\usepackage[normalem]{ulem}
\usepackage[utf8]{inputenc}
\usepackage{amsmath, amssymb, amsthm, enumerate,url,bibentry, color,xcolor, fullpage,placeins,hyperref}

\usepackage[all,cmtip]{xy}
\usepackage{tikz,epsfig}
\usepackage{verbatim}
\usepackage{geometry}
\usepackage{colonequals}
\usepackage{tabularray}
\usepackage{cleveref}
\usepackage[singlelinecheck=false]{caption}
\usepackage{tikz-cd}

\definecolor{purple}{rgb}{0.59, 0.44, 0.84}

\definecolor{orange}{rgb}{1, 0.6, 0.4}

\newtheorem{theorem}{Theorem}[section]
\newtheorem{lemma}[theorem]{Lemma}

\newtheorem{proposition}[theorem]{Proposition}

\newcounter{example}
\newenvironment{example}[1][]{\refstepcounter{example}\par\medskip
   \noindent \textbf{Example~\theexample. #1} \rmfamily}{\medskip} 
\newcounter{remark}
\newenvironment{remark}[1][]{\refstepcounter{remark}\par\medskip
   \noindent \textbf{Remark~\theremark. #1} \rmfamily}{\medskip}

\numberwithin{equation}{section}
\numberwithin{example}{section}
\numberwithin{remark}{section}

\def \l {\lambda}

\def\BG{\mathtt G}
\def\C{\mathbb C}
\def\G{\mathbf G}

\def\wp{\mathfrak{p}}

\def\Gfam{\mathfrak{G}}

\def\Tr{{\rm Tr}}
\def\Gal{{\rm Gal}}
\def\Frob{{\rm Fr}}
\def\F{{\mathbb F}}
\def\Z{{\mathbb Z}}

\def\G{\Gamma}

\def\SL{\mathrm{SL}}

\def\M#1#2#3#4{\begin{pmatrix}#1&#2\\#3&#4\end{pmatrix}}

\def\R{\mathbb{R}}
\def\Q{\mathbb{Q}}

\def\H{\mathbb H}
\def\M#1#2#3#4{\begin{pmatrix}#1&#2\\#3&#4\end{pmatrix}}

\def \ol{\overline}

\newcommand{\fq}{\mathbb{F}_q}

\newcommand{\fqhat}
{\widehat{\mathbb{F}_{q}^{\times}}}

\newcommand*\HYPERskip{&}
\catcode`,\active
\newcommand*\pFq{
\begingroup
\catcode`\,\active
\def ,{\HYPERskip}
\doHyper
}
\catcode`\,12
\def\doHyper#1#2#3#4#5{
\, _{#1}F_{#2}\left[\begin{matrix}#3 \smallskip \\  #4\end{matrix} \; ; \; #5\right]
\endgroup
}

\def \pfq{\pFq}

\catcode`,\active

\catcode`\,12

\crefformat{section}{\S#2#1#3}
\crefformat{subsection}{\S#2#1#3}

\catcode`,\active

\catcode`\,12

\catcode`,\active

\catcode`\,12

\def\SL{\mathrm{SL}}

\def\M#1#2#3#4{\begin{pmatrix}#1&#2\\#3&#4\end{pmatrix}}

\def \HD{\mathit{HD}}

\def\BH{\mathbf H}

\def\G{\Gamma}
\begin{document}
\title{The Explicit Hypergeometric Modularity Method III}

\address{Department of Mathematics, West Virginia University Institute of Technology, Beckley, WV 25801, USA}

\email{\href{mailto:brian.grove@mail.wvu.edu}{brian.grove@mail.wvu.edu}}

\address{Department of Mathematics, Louisiana State University, Baton Rouge, LA 70803, USA}

\email{\href{mailto:llong@lsu.edu}{llong@lsu.edu}}

\address{Department of Mathematics, Cornell University, Ithaca, NY 14853, USA}

\email{\href{mailto:esmerosen92@gmail.com}{esmerosen92@gmail.com}}

\address{Department of Mathematics, Louisiana State University, Baton Rouge, LA 70803, USA}

\email{\href{mailto:ftu@lsu.edu}{ftu@lsu.edu}}



\begin{abstract}
 We refine the Explicit Hypergeometric Modularity Method (EHMM) and
develop a variant that applies to a broader class of hypergeometric
data. As an application, we establish the modularity of hypergeometric
Galois representations arising from length-$4$ data that are not
necessarily defined over $\Q$. We also use this method to give explicit
constructions of nine modular forms associated with hypergeometric
rigid Calabi--Yau threefolds conjectured to be modular by
Rodriguez-Villegas. The modularity of these threefolds was first proved
by Long--Tu--Yui--Zudilin in \cite{LTYZ} using a different approach
based on the Faltings--Serre method. Moreover, the explicit nature of the method makes it well suited to
the computation of special  $L$-values of the associated modular forms.
\end{abstract}
\author{Brian Grove, Ling Long, Esme Rosen, and Fang-Ting Tu}
\maketitle

\setcounter{tocdepth}{1}
\tableofcontents

\section{Introduction}

In \cite{HMM1}, the Explicit Hypergeometric Modularity Method (EHMM) was introduced to study the modularity of certain hypergeometric Galois representations. The method combines commutative formal group laws (CFGLs) with Ramanujan’s theory of elliptic functions to alternative bases (REAB), which provides an explicit bridge between certain hypergeometric functions and modular forms.

The Langlands program predicts the modularity of motivic $G_\Q:=\Gal(\ol \Q/\Q)$-representations. In \cite{LLL26}, potential automorphy has been established for irreducible hypergeometric Galois representations of $G_\Q$ with distinct Hodge–Tate weights, as well as in several low-dimensional hypergeometric cases. However, determining the exact automorphic correspondence remains challenging. Traditional approaches, such as the Faltings–Serre method, often rely on searches in existing databases, while methods based on the Eichler–Selberg trace formula, as in \cite{Ahlgren-Ono-CalabiYau, lennon1, McCarthy-Papanikolas}, tend to be 
case-dependent. By design, the EHMM provides a general algorithm that applies uniformly to a broad range of hypergeometric data and computes the corresponding Hecke eigenforms explicitly.

 Let  $\HD$ or simply $\varkappa$  be an arithmetic hypergeometric datum consisting of
\[
\left\{
\alpha=\{a_1,\ldots,a_n\},
\quad
\beta=\{b_1:=1,\ldots,b_n\}
\right\}
\]
i.e. a pair of rational multisets with equal length. Let $M$ be the least common denominator of all $a_i$ and $b_j$. For two  hypergeometric data $\varkappa_1=\{\alpha_1,\beta_1\}$ and $\varkappa_2=\{\alpha_2,\beta_2\}$, we define 
$$
\varkappa_1\star\varkappa_2=\{\alpha_1\cup \alpha_2,\beta_1\cup \beta_2\}. 
$$ 

The main objective of the EHMM is to explicitly construct a Hecke eigenform or automorphic form $f_{\HD}^{\sharp}$ associated with the representation $\eta_{\HD, \ell,1}$ of the Galois group $\operatorname{Gal}\bigl(\overline{\mathbb{Q}}/\mathbb{Q}(\zeta_M)\bigr)$, arising from $\HD$ at the parameter $t=1$, whose Frobenius traces are expressed in terms of the finite field hypergeometric functions $\BH(\alpha,\beta;1;-)$. Here we follow the notation of \cite{LLL26}. The general philosophy of the EHMM is to decompose the datum as
\[
\HD=\HD^{a}\star HD^\flat,
\]
for data $\HD^{a}$ and $\HD^\flat$, to achieve the following goals. 
The datum $\HD^{\flat}$ typically carries a REAB background and determines an underlying modular form $f$, possibly requiring further refinement, whereas the augmenting datum $\HD^{a}$ is used to adjust either the weight of $f$ or its vanishing behavior at the cusps. The desired outcome is a holomorphic cusp form $f_{\HD}$. Since the publication of the first paper in the series, \cite{HMM1}, the EHMM has been applied in several subsequent works, including \cite{HMM2,Barman_MaityEHMM,grove2025hypergeometricmodularityconjecturesdawsey,rosen, rosenK1}. We have also implemented an EHMM calculator, available at \cite{EHMMcalc}. 

In \cite{HMM1,Barman_MaityEHMM,  grove2025hypergeometricmodularityconjecturesdawsey,rosen,  rosenK1}, one takes the length of $HD$ less than $5$ and considers the Euler integral representative  of  $F(HD,1)=\pFq{n}{n-1}{a_{1}&a_{2}&\ldots&a_{n}}{&b_{2}& \ldots &b_{n}}1$. That is, 
$ \HD^\flat=\{\{a_1,\cdots, a_{n-1}\},\{1,b_2,\cdots, b_{n-1}\}\}$ and    $\HD^{a}=\bigl\{\{a_n\},\{b_n\}\bigr\}$. 
Once the modular form $f_{HD}$ has been constructed, one checks whether 
the hypergeometric datum satisfies the regularity condition, such as that of 
\cite{HMM1,rosen}. When this condition holds, the EHMM in \cite{HMM1} describes 
$\eta_{HD,\ell,1}$ through a system of congruences. Truncated hypergeometric series play a central role in this argument. 
They connect
$\BH(\alpha,\beta;1;\mathbb{F}_p)$
with the $p$-th Fourier coefficient $a_p(f_{\HD})$
of $f_{\HD}$. Sufficiently strong congruences, combined with Weil bounds, 
allow these congruences to be upgraded to equalities.

The principal goal of this paper is to demonstrate a variant of the EHMM based on the residue theorem,
\begin{equation}\label{eq:star-at1}
    F(HD^a\star HD^\flat,1) = \frac 1{2\pi i} \oint_{|t|=1} F(HD^a,1/t)F(HD^\flat,t)\frac{dt}{t}.
\end{equation}
\begin{quote}
   \emph{ 
   By this approach, we develop  a variant of EHMM which is applicable to $\HD^a$ of length larger than 1.}
\end{quote}   Another feature is
the extension of the EHMM to primes that are not necessarily congruent
to $1$ modulo $M$, which was required in \cite{HMM1}.

As an illustration, for integers $d,e>1$, we consider the 
hypergeometric data
\begin{equation}\label{eq:keydata}
    \varkappa(d,e):=\left\{\alpha=\left\{\tfrac1d,\tfrac{d-1}d,\tfrac1e,\tfrac{e-1}e\right\},\beta=\{1,1,1,1\}\right\}
\end{equation}
and use the decomposition  $\varkappa(d,e):=HD^a\star HD^\flat$ where 
\begin{equation}\label{eq:partition}
   HD^\flat=\left\{\left\{\tfrac12,\tfrac1d,\tfrac{d-1}d\right\},\{1,1,1\}\right\}\quad \text{and} \quad HD^a=\left\{\left\{\tfrac1e,\tfrac{e-1}e\right\},\left\{1,\tfrac12\right\}\right\}.
\end{equation} 

\begin{theorem}\label{thm:H-version}
Let $d\in\{2,3,4\},$
and define $M_2=8,\, M_3=12,\, M_4=24,$
and $C_1(2)=-64,\, C_1(3)=108,\, C_1(4)=256.$
For each positive divisor $e$ of $M_d$,   there exists a vector space $\mathfrak G_{d,e}$ of weight-$4$
cusp forms constructed using the EHMM, of dimension equal to
Euler's totient of $e$, such that for every normalized Hecke eigenform $g_{d,e}\in\mathfrak G_{d,e}$ and every prime
$
p\equiv 1\pmod{\text{lcm}(d,e)}
$
 coprime to $a_p(g_{d,e})$, one has
\[
\BH\left(
\varkappa(d,e);1;\wp
\right)
=
\varepsilon_{d,e}(p)  a_p(g_{d,e})
+
(-1)^{\frac{p-1}{d}+\frac{p-1}{e}}p,
\]
where  $a_n(g)$  denotes the $n$th Fourier coefficient of $g$,  and $\varepsilon_{d,e}(p)\in\{\pm1\}$ is  characterized by the congruence
\[
\varepsilon_{d,e}(p)\equiv\left(-\frac{C_1(d)}{4}\right)^{\frac{1-p}{e}}
\pmod p.
\]
\end{theorem}

\smallskip

Throughout this paper, for $c\in\mathbb Q^\times$, we use $\chi_c$ to denote the  Dirichlet
character corresponding to the character of $G_\Q$
whose kernel fixes $\mathbb Q(\sqrt c)$.

As an application, after extending the method to include additional primes,
we explicitly obtain nine target modular forms associated with
hypergeometric rigid Calabi--Yau threefolds
\cite{gugiatti2024hypergeometriclocalsystemsmathbbq,LTYZ,RV-conj}.

\begin{theorem}\label{thm: Galois version}
Let $d=2$, $3$, $4$, and $M_d=8$, $12$, $24$, respectively. Assume that
$e\in\{2,3,4,6\}$ and that $e\mid M_d$. Let
$\eta_{\varkappa(d,e),\ell,1}^{\mathrm{BCM}}$ denote the
semisimplification of the hypergeometric representation of
$G_{\Q}$ associated with
$\varkappa(d,e)$ as defined in \cite{LLL26}. 

Let $g_{d,e}$
be the weight-$4$ Hecke eigenform listed in
Tables~\ref{G2table}, \ref{G3table}, and \ref{G4table}.
When $d=2$, let $f_{2,e}^\sharp:=g_{2,e}$;  and when $d=3$ or $4$,
$f_{d,e}^{\sharp}:=\chi_{-d}\otimes g_{d,e}$. 
Then
\[
\eta_{\varkappa(d,e),\ell,1}^{\mathrm{BCM}}
\cong
\rho_{f_{d,e}^{\sharp}}
\oplus
\left(
\varsigma(d)\varsigma(e)
\otimes\epsilon_{\ell}
\right),
\]
where $\rho_{f_{d,e}^{\sharp}}$ denotes the representation of $G_{\Q}$
attached to $f_{d,e}^{\sharp}$, and $\epsilon_{\ell}$ denotes the
$\ell$-adic cyclotomic character. For $e=2,3,4,6$, $\varsigma(e)$ is
$\chi_{-1},\chi_{-3},\chi_{-2},$ and $\chi_{-1}$, respectively. 
\end{theorem}

\begin{remark}\label{rem:1}
Note that, for $e=2,3,4,6$, the sign
$(-1)^{\frac{p-1}{e}}$ in Theorem \ref{thm:H-version} coincides, for primes $p\equiv 1\pmod e$, with
the quadratic Dirichlet characters $\varsigma(e)$. 
\end{remark} 

For example, when $(d,e)=(4,2)$, the EHMM yields
\begin{equation}\label{eq:explicitRV}
f_{4,2}^{\sharp}
=
\frac{\bigl(\eta(2\tau)\eta(4\tau)\bigr)^{12}}
{\eta(2\tau)^{24}+64\eta(4\tau)^{24}}
\left(E_2(2\tau)-2E_2(4\tau)\right)^2.
\end{equation}
This is the form $f_{16.4.a.a}$ in the notation of the $L
$-functions and
Modular Forms Database (LMFDB), which we adopt for all Hecke eigenforms
in this paper. Here, $\eta(\tau)=q^{1/24}\prod_{n=1}^{\infty}(1-q^n)$
is the Dedekind eta function, and
$E_2(\tau)=1-24\sum_{n,d\geq 1}nq^{dn}$ is the normalized Eisenstein
series of weight $2$. Thus, \eqref{eq:explicitRV} gives an explicit
construction of the Hecke eigenform appearing in \cite{LTYZ}. This
eigenform was originally identified  in \cite{Fuselier-McCarthy} and later studied in \cite{LTYZ} using the Faltings--Serre method together with the modularity of certain rigid Calabi--Yau threefolds.

Further details on
the EHMM and its variant are presented in
\S\ref{sec:EHMM-variation}, together with an outline of the proofs. The
method is then used to construct the families $\mathfrak G_{d,e}$ from REAB in \S\ref{constructingMF}
appearing in both main theorems. \S\ref{ss:CFGL} recalls the
CFGL framework which is used to establish congruences
between hypergeometric coefficients and Fourier coefficients, and thus the modularity of $\eta_{\varkappa(d,e),\ell,1}^{\mathrm{BCM}}$ . The
proofs of the main theorems are given in \S\ref{ss:modularity}.
Further questions, including applications to the computation of special
$L$-values of the corresponding modular forms, are discussed in
\S\ref{sec:future} and \S\ref{lvalueappendix}. The paper concludes with an appendix,
\S\ref{EHMMappendix}, listing the Hecke eigenforms $g_{d,e}$. A list of notation appears in \S\ref{sec:notation}.

\section*{Acknowledgments}
The second author was supported by the Simons Foundation grant \#MP-TSM-00002492 and the LSU Michael F. and Roberta Nesbit McDonald Professorship. The fourth author was supported by  NSF Grant DMS 2302531.

\section{EHMM and one of its variations}\label{sec:EHMM-variation}

The EHMM combines techniques from REAB, $p$-adic analysis, and the symmetries of hypergeometric series to construct explicit families of modular forms from input on the hypergeometric side and to establish the modularity of the corresponding hypergeometric Galois representations. 
 
Consider $$\varkappa = \{ \{a_{1}, a_{2}, \ldots, a_{n}\}, \{b_{1}=1, b_{2} \ldots, b_{n}\}\},$$ a hypergeometric datum of rational numbers so that $0<a_i,b_i\leq 1$.  Define the \textit{conjugate data}
\begin{equation}\label{eq:conjdata}
c \cdot \varkappa:= \{\{c \cdot a_{1},c \cdot a_{2}, \ldots, c \cdot a_{n}\}, \{1,c \cdot b_{2}, \ldots, c \cdot b_{n}\}\}
\end{equation}
for $c \in (\Z/M\Z)^{\times}$ with all entries interpreted modulo $\Z$.

 The hypergeometric character sums used in this paper are the traces of Frobenius for the Galois representation $\eta_{\varkappa,\ell,t}$, a twist of the initial version of hypergeometric Galois representations first defined by Katz \cite{Katz-Dwork,KatzESDE}. Refer to \cite{Win3X} for a survey of the various twists of hypergeometric Galois representations. We use the hypergeometric character sums in \eqref{eq:Hpdef}, originally defined by McCarthy \cite{McCarthy}.
 The setup is as follows. 
 
 Let $M$ be the least common denominator of the entries in the datum $\varkappa$. Then for each prime ideal $\mathfrak{p}$ of $\Z[\zeta_M]$ coprime to $M$ we can associate to the residue field $\kappa_{\mathfrak{p}}:= \Z[\zeta_{M}]/\mathfrak{p}$ of size $q$ a character via the $M$-th residue symbol, \begin{equation}\label{eq:iotadef}
\iota_{\mathfrak{p}}\left(\frac{i}{M}\right)(x):= \left(\frac{x}{\mathfrak{p}}\right)_{M}^{i} \equiv x^{(q-1)\frac{i}{M}} \pmod{\mathfrak{p}}
\end{equation} for all $x \in \Z[\zeta_{M}]$. We also use the Gauss sum, for  $\chi \in \fqhat$,  $$g(\chi) = \sum_{x \in \fq^{\times}} \chi(x) \zeta_{p}^{\textrm{Tr}(x)}.$$ 
Now define the hypergeometric character sums
\begin{equation}\label{eq:Hpdef}
\BH(\varkappa;t;{ \wp}) := \frac{1}{1-q} \sum_{\chi \in \widehat{\kappa_{\mathfrak{p}}^{\times}}} \chi((-1)^{n}t) \prod_{i=1}^{n} \frac{g(\iota_{\mathfrak{p}}(a_{j})\chi)g(\iota_{\mathfrak{p}}(-b_{j})\overline{\chi})}{g(\iota_{\mathfrak{p}}(a_{j}))g(\iota_{\mathfrak{p}}(-b_{j}))},
\end{equation}
where $t \in \kappa_{\mathfrak{p}}^{\times}$, and $g(\cdot)$ is the Gauss sum on $ \widehat{\kappa_{\mathfrak{p}}^{\times}}$. 
When $\BH(\varkappa;t; \wp)$ is independent of $\wp$ and its  conjugates, we write it as $\BH(\varkappa;t;\F_q)$ where $q=|\Z[\zeta_M]/\wp|$. 

An important perspective underlying our proofs is to regard, for a given $t\in\Q$, $\BH(\varkappa;t;{ \wp})$ as the Frobenius traces of the hypergeometric Galois representation $\eta_{\varkappa,\ell,t}$ of the absolute Galois group of $\Q(\zeta_{M})$, as alluded to in the introduction. Let
\begin{equation}\label{eq:Sta-HD}
   St_{\varkappa}=\{c\in (\Z/M\Z)^\times \mid c\cdot \varkappa=\varkappa\}.
\end{equation}
Let $K_\varkappa$ denote the fixed field of $\Q(\zeta_M)$ under $St_{\varkappa}$, where, as usual, we identify $\Gal(\Q(\zeta_M)/\Q)$ with $(\Z/M\Z)^\times$. The datum is then said to be \textit{defined over $K_\varkappa$}. This field plays an important role in hypergeometric arithmetic. In terms of $\BH(\varkappa;t;{ \wp})$, this means that their values lie in $K_\varkappa$. Equivalently, if $\varkappa$ is defined over $K_\varkappa$, then $\eta_{\varkappa,\ell,t}$ can be extended to $G_{K_\varkappa}$.

In the special case where $K_\varkappa=\Q$, i.e. the datum is defined over $\Q$, the work of Beukers–Cohen–Mellit \cite{BCM} provides a fixed choice of an extension of $\eta_{\varkappa,\ell,t}$ to $G_\Q$, which we denote by $\eta_{\varkappa,\ell,t}^{BCM}$, as mentioned in Theorem \ref{thm: Galois version}. The corresponding Frobenius traces are given by an extension of \eqref{eq:Hpdef} through an equivalent form that is applicable to almost all primes of $\Z$. One can check that the condition $e\in \{2,3,4,6\}$ is equivalent to requiring that $\varkappa(d,e)$ be defined over $\Q$.

Hypergeometric Galois representations are expected to to be the \'etale realization of a \textit{hypergeometric motive}, whose formal definition has not yet been written in the literature, to our knowledge. However, the expected realizations of these motives are better understood and are discussed in 
\cite{Fedorov18,fu-li-wan-GKZ-padic,KatzRigid,HypMot}. The heart of the EHMM seems to be the comparison between the $\ell$-adic and $p$-adic aspects of hypergeometric motives. 

First, define the classical hypergeometric series
 \begin{equation*}
 F(\varkappa,t) := \sum_{k \geq 0} A_{\varkappa}(k)t^{k}, \quad \textrm{where} \quad A_{\varkappa}(k) = \prod_{i=1}^{n} \frac{(a_{i})_{k}}{(b_{i})_{k}}, 
 \end{equation*}
 and $(a)_{k} = a(a+1) \cdots(a+k-1)$ is a Pochhammer symbol with $b_1$ typically assumed to be 1. Then define the truncated hypergeometric series as \begin{equation*}
 F(\varkappa,t)_{m} := \sum_{k=0}^{m} A_{\varkappa}(k)t^{k}.
 \end{equation*} 
 
 By the work of Dwork \cite{Dwork}, if the $b_i$ are all equal to 1, there is a well-defined $p$-adic unit root $\mu_{\varkappa,p,t}$, whose first $p$-adic digit can be computed using $F(\varkappa,t)_{p-1}$. In modern language, this unit root can be interpreted as a root of the characteristic polynomial of Frobenius acting on crystalline cohomology. Refer to \cite{kedlaya} for a modern exposition. Due to comparison isomorphisms between \'etale and crystalline cohomology, one anticipates that the characteristic polynomials of Frobenius for \'etale and crystalline cohomology should coincide at primes $p$ not dividing $M$ or the numerator/denominator of $t$. As there is a unique unit root by Dwork's theory, we anticipate that $\mu_{\varkappa,p,t}$ and the trace of Frobenius, which is the sum of the roots, should coincide modulo $p$. Such congruences are referred to as Gross--Koblitz-type congruences in \cite{HMM1}. 

Work in \cite{HMM1,LLL26} generalizes Dwork's theory to cases where the $\beta_i$ are not all equal to 1, assuming a sufficient regularity condition which ensures there is exactly one unit root. They also compute the value of $\BH(\varkappa,t,\wp)$ modulo $p^2$, establishing supercongruences in many cases. In \cite{HMM1,rosen}, the regularity is given via conditions on the parameter set. The regularity is reformulated as a geometric condition involving Hodge numbers in the forthcoming paper, \cite{rosenmixedI}. Assuming regularity, such results can be proved using the Gross-Koblitz formula \cite[Theorem 2.5]{HMM1}, which relates the Gauss sums in the $\BH$ function with the coefficients of the truncated hypergeometric series.  In any case, the conclusion is that the $\BH(\varkappa;1;\wp)$ and $F(\varkappa,1)_{p-1}$ values are congruent modulo $p$, up to a fixed embedding of $\zeta_M$ in the $p$-adic ring and an explicit normalizing factor, and this enables the EHMM to connect $p$-adic analysis to Galois representations.

 The relation with modular forms is produced via the theory of CFGL,  
 relating the truncated series $F(\varkappa,1)_{p-1}$ and certain $a_{p}(f_{\varkappa})$ coefficients via REAB. The idea is as follows. Begin with the Euler integral formula \cite[Equation (3.2)]{Win3X} for the hypergeometric series $F(\varkappa,1)$:

 \begin{equation}\label{eq:eulerint}
 F(\varkappa,1) = \frac{\G(b_{n})}{\G(a_{n})\G(b_{n}-a_{n})} \cdot \int_{0}^{1} t^{a_{n}}(1-t)^{b_{n}-a_{n}-1} F(\varkappa^{\flat},t) \frac{dt}{t},
 \end{equation}
 where $\Gamma(\cdot)$ is the gamma function,
 and for any datum $\varkappa = \{ \{a_{1}, a_{2}, \ldots, a_{n}\}, \{b_{1}=1, b_{2} \ldots, b_{n}\}\}$  and default choice of $\varkappa^{\flat}$ in \cite{HMM1} is  $$\varkappa^{\flat} = \{\{a_{1}, \ldots, a_{n-1}\},\{b_{1} = 1, \ldots, b_{n-1}\}\}.$$ 
Then consider the differential
 \begin{equation}\label{eq:integrand}
 f_{\varkappa}(t) := t^{a_{n}}(1-t)^{b_{n}-a_{n}-1}F(\varkappa^{\flat},t) \frac{dt}{t}.
 \end{equation}
Now the important work of Ramanujan \cite{Ramanujan-pi} and later developments, such as \cite{BBG-Ramanujan,BB}, show that choosing $t = t(\tau)$ to be an appropriate modular function makes $f_{\varkappa}(t)$ a modular form, in certain cases.  The family of examples considered in \cite{HMM1} involves the data
 $$\varkappa(d;r,s):= \{\{1/d,(d-1)/d,r\},\{1,1,s\}\}$$
 for $d \in \{2,3,4\}$ and $(r,s)$ in appropriate subsets of $\Q^{2}$.

 The CFGL, see for example \cite{Stienstra-Beukers}, then provides  
 explicit congruences 
 between the coefficients of the $t$-expansion  
 and the $q$-expansion (obtained by substituting a suitable modular function for $t$) for $f_{\varkappa}(t)$.

 We now discuss the CFGL via the $t$ and $q$ expansions of $F(\varkappa(d;r,s)^{\flat},t)$. The $t$-expansion is obtained by writing the differential of \eqref{eq:eulerint} as a formal power series in $t$. This expansion has terminating $F$ functions as coefficients in the indexing variable $k$; see \Cref{lem:HG-product} for a general result and \eqref{eg:1} for the case $\varkappa(2;1/2,1/2)$. Then appropriate choices of $k$ that are $p$-adically close to rational numbers and a careful analysis of hypergeometric coefficients modulo $p$ give congruences between the terminating and truncated series produced from the $t$-expansion. See the discussions in \Cref{ss:CFGLhere} for more details.
 
 In the $q$-expansion, choosing $t = t(\tau)$ makes $F(\varkappa(d;r,s)^{\flat},t(\tau))$ 
 into a weight one modular form for $\Gamma_{1}(6-d)$, see \cite{StillerHF, Yang04}. In particular, $f_{\varkappa(d;r,s)}$ is a weight three modular form on an appropriate finite-index subgroup of $\Gamma_{1}(6-d)$. 
 For example,  consider the $d=2$ and $d=3$ cases. The $d=2$ case relies on the modular lambda function
$$\lambda(\tau) = 16 \frac{\eta(\tau/2)^{8}\eta(2 \tau)^{16}}{\eta(\tau)^{24}},$$
a Hauptmodul for $\Gamma(2)$,  conjugate to $\G_1(4)$, the associated weight $1/2$ Jacobi theta functions 

 $$\theta_{2}(\tau) = \sum_{n \in \Z} q^{(2n+1)^{2}/8}, \quad \theta_{3}(\tau) = \sum_{n \in \Z} q^{n^{2}/2}, \quad \textrm{and} \quad \theta_{4}(\tau) = \sum_{n \in \Z} (-1)^{n}q^{n^{2}/2},$$
 and
\begin{equation}\label{eq:d=2eval}
 F(\varkappa(2;r,s)^{\flat},\lambda(\tau)) = \theta_{3}^{2}(\tau).
 \end{equation}
 This setup leads to 
 the weight three cusp forms \cite{HMM1,HMM2,rosen}
 $$\mathbb{K}_{2}(r,s)(\tau):= \eta(\tau/2)^{16s-8r-12}\eta(\tau)^{30-24s}\eta(2 \tau)^{8s+8r-12}$$
 and the weight two $\mathbb{K}_{1}$ functions \cite{rosenK1}. In a parallel manner, for $d = 3$, the Hauptmodul 
 $$t_{3}(\tau) = 27 \frac{\eta(3 \tau)^{9}}{\left(3\eta(3 \tau)^{3}+\eta(\tau/3)^{3}\right)^{3}}$$
 for $\Gamma_{1}(3)$, the associated weight one cubic theta functions \cite{Borweincubic,BBG} and REAB lead to the weight three cusp forms \cite{grove2025hypergeometricmodularityconjecturesdawsey}
 $$\mathbb{K}_{3}(r,1)(\tau):= \eta(\tau)^{9-12r}\eta(3 \tau)^{12r-3}$$ 
 and the associated weight two $\mathbb{K}_{5}$ functions \cite{Barman_MaityEHMM}.

Further, in a small number of cases, Clausen formula for hypergeometric series combine with identities such as \eqref{eq:d=2eval}, to produce weight four modular forms. For example, in \cite{Barman_MaityEHMM} the authors use \eqref{eq:d=2eval} to construct the $\mathbb{K}_{4}$ family of weight four eta-quotients when $d=2$ for the data 
$$\{\{1/2,1/d,1-1/d,r\},\{1,1,1,r+1/2\}\},$$
with $24r \in \Z^{+}$. See the appendix of \cite{HMM1} for more details. When $d=3,4$, the modular forms  consist of a modular function
multiplied by the square of a weight-two  
Eisenstein series 
$E_{2,N}(\tau):= \frac1{N-1}\left(NE_2(N\tau)-E_2(\tau)\right).$  Precisely, using the modular functions in \Cref{tab:bigger triangle}, they are 
\begin{equation}\label{eqn: Euler}
t_{d,+}^r ( E_{2,6-d})^2, \quad d =3, 4, 
\end{equation}
where $r=i/4$ for $d=4$, and $r=i/6$ for $d=3$. 

     An important aspect of the EHMM is that the cusp forms constructed from REAB and hypergeometric identities can be completed to Hecke eigenforms, in many cases \cite{HMM1,HMM2, Barman_MaityEHMM, grove2025hypergeometricmodularityconjecturesdawsey, rosen}. These eigenforms are built from the conjugates
of the hypergeometric data.

In this article, we use the partition \eqref{eq:partition} to construct the weight four cusp forms $\BG_{d}(i/e)(e \tau)$ in \eqref{eq:g(d,e,i)} for $d \in \{2,3,4\}$. For example, consider $ \varkappa(4,2)$ from \eqref{eq:keydata} and \eqref{eq:explicitRV}. Then the conjugates are $\varkappa(4,2)$ and $3 \cdot \varkappa(4,2) = \varkappa(4,2)$. At the level of $\BG_{2}$ functions, this conjugate data corresponds to the choices $1/4$ and $3/4$.  There are eigenforms which are linear combinations of the $\BG_{2}(1/4)(4 \tau)$ and $\BG_{2}(3/4)(4 \tau)$:
$$f_{16.4.a.a} = \BG_{2}(1/4)(4 \tau) + 4 \BG_{2}(3/4)(4 \tau)$$
and 
$$f_{8.4.a.a} = \BG_{2}(1/4)(4 \tau) - 4 \BG_{2}(3/4)(4 \tau). $$
See \S \ref{constructingMF} for more details on the cusp forms $\BG_{d}(i/e)(e \tau)$ and the associated families $\Gfam_{d,e}$. 
We now provide more detail on the variation of the EHMM used throughout.

\subsection{Over \texorpdfstring{$\C$}{C} background of the variation}
We first discuss the hypergeometric background that corresponds to the partition  $\HD^a\star \HD^\flat$  at argument $t=1$ using \eqref{eq:star-at1}. 
This idea was used by Zagier in \cite{Zagier-top-diff} and we refer to it as  Zagier's trick.

We will focus on the cases with 
\begin{equation*}
 HD^a=\varkappa_{alg}(r):=\{\{r,1-r\},\{1,1/2\}\}, \quad \mbox{and} \quad HD^\flat=\varkappa_3(d), \quad \mbox{for } d=2, 3, 4, 
\end{equation*}
where $\varkappa_3(d)=\{\{1/d,1-1/d,1/2\},\{1,1,1\}\}$
and $r$ is certain rational number.

To proceed with our discussion making use of formal power series in \S \ref{ss:CFGL}, we will first express the series $F(\varkappa_{alg}(a),1/t)$ in $1/t$ to a power series in $t$ by considering its analytic continuation from $\infty$ to $0$. It can be done by specializing \cite[Theorem 2.3.2]{AAR}  to $b=1-a$ and $c=\frac12$. Then,
\begin{align*}
    F(\varkappa_{alg}(a), 1/t)&= (-1)^{-a}\frac{\G(\frac12)\G(1-2a)}{\G(a)\G(\frac12-a)}t^{a} F(\varkappa_{cont}(a), t)\\& \quad +(-1)^{a-1} \frac{\G(\frac12)\G(2a-1)}{\G(a)\G(a-\frac12)}t^{1-a} F(\varkappa_{cont}(1-a), t),    
 \end{align*} where $\varkappa_{cont}(a):=\{\{a, a+1/2\},\{1, 2a\}\}$. 
 Thus from
the combination of the  algebraic formula \cite[Equation 15.1.14]{handbook}  $$
 F(\varkappa_{cont}(a), t)= \frac1{\sqrt{1-t}}\left(\frac{1+\sqrt{1-t}}2\right)^{1-2a},
$$ and  $\G(z)\G(z+\tfrac12) = 2^{1-2z}\G(\tfrac12)\G(z)$,  \cite[Theorem 2.3.2]{AAR}  tells us
 \begin{align*}    
 F(\varkappa_{alg}(a), 1/t)   &=(-1)^{-a} (t/4)^{a} F(\varkappa_{cont}(a), t)+(-1)^{a-1} (t/4)^{1-a} F(\varkappa_{cont}(1-a), t)\\
   &=(-1)^{-a}  A_a(t)+(-1)^{a-1} \frac14\frac t{(1-t)}\frac1{A_a(t)}, 
\end{align*} 
where $A_a(t)= (t/4)^{a} \frac1{\sqrt{1-t}}\left(\frac{1+\sqrt{1-t}}2\right)^{1-2a}$. Note that we pick the principal branch for all square roots and other radicals.

By Clausen's formula \cite[Section 3.1]{AAR},  when $|4z(1-z)|\leq 1$,
\begin{align}\label{eq:Clausen}
   F(\varkappa_2(d),z)^2= F(\varkappa_3(d),4z(1-z)), \text{ where } \varkappa_2(d)=\{\{1/d,1-1/d\},\{1,1\}\}.
\end{align}
When fixing $z$  such that  
$t=4z(1-z)$. 
One has  
$$
F(\varkappa_{alg}(a), 1/4z(1-z))=\tfrac{1-z}{1-2z}\left(\left(\tfrac z{z-1}\right)^a+\left(\tfrac z{z-1}\right)^{1-a}\right);
$$
Zagier's trick  essentially becomes 
\begin{equation}\label{eq:Zagier}
    F\left(\varkappa_{alg}(a),\tfrac 1t\right)F(\varkappa_3(d),t)\tfrac{dt}{t}=\left(\left(\tfrac z{z-1}\right)^a+\left(\tfrac z{z-1}\right)^{1-a}\right)F(\varkappa_2(d),z)^2 \tfrac{dz}{z}. 
\end{equation} 
The identity leads to our constructions of hypergeometric-type modular forms in \S \ref{constructingMF}.

\subsection{Modular background}
The monodromy group of the differential equation satisfied by $F(\varkappa_{alg}(a), t)$ is a finite Dihedral group \cite{Win3X}.  
For each $d$, the monodromy group corresponds to  $F(\varkappa_3(d),t)$ is isomorphic to the triangle group $(2,\tfrac{2d}{d-2},\infty)$ when reading $1/\infty =0$. Such a group has an index-two triangle subgroup  $(\tfrac{d}{d-2},\infty, \infty)$, isomorphic to the monodromy group corresponding to the datum $\varkappa_2(d)$.  We will fix our choice of  
$(2,\tfrac{2d}{d-2},\infty)$ commensurable with $\SL_2(\Z)$  and accordingly $(\tfrac{d}{d-2},\infty, \infty)$ is isomorphic to $\G_1(6-d)$ when identifying $\pm 1$ as the trivial action on the upper half-plane $\H$. 
All the relevant information is collected in  \Cref{tab:bigger triangle} and \Cref{table: Index2}, see \cite[Appendix II]{HMM1} for example.

\renewcommand{\arraystretch}{1.3}
\begin{table}[ht]
\begin{center}
$$
\begin{array}{c|ccccccc}
 d  &\left(2,\tfrac{2d}{d-2},\infty\right)&t_{d}&C_1(d) \\ \hline
 2&\G_0(2)\simeq(2,\infty,\infty)&u(\tau)=-64\frac{\eta(\tau)^{24}}{\eta(2\tau)^{24}}&-64
 \\
 3    & \G_0(3)^+\simeq(2,6,\infty) &t_{6+}= \frac{108\eta(\tau)^{12}\eta(3\tau)^{12}}{\left(\eta(\tau)^{12}+27\eta(3\tau)^{12}\right)^{2}}&108 
 \\
 4&\G_0(2)^+\simeq(2,4,\infty) &t_{4+}=\frac{256\eta(\tau)^{24}\eta(2\tau)^{24}}{\left(\eta(\tau)^{24}+64\eta(2\tau)^{24}\right)^{2}} &256
\end{array}
$$ 
  \caption{The triangle groups $(2,\tfrac{2d}{d-2},\infty)$ with the chosen Hauptmoduln $t_d$ with leading $q$-coefficients $C_1(d)$}
    \label{tab:bigger triangle}
 \end{center}  
\end{table}

\begin{table}[ht]
 \begin{center}
$$
\begin{array}{c|ccccccc}
 d  &  \G_1(6-d)& z_d &\frac{z_d}{z_d-1}&F(\varkappa_2(d),z_d)^2& \Theta z_d\\ \hline
 2  & \G_1(4)&\frac{\l(2\tau)}{\l(2\tau)-1}&\l(2\tau)&\theta_4(2\tau)^4&\theta_3(2\tau)^4=E_{2,4}(\tau)\\
 3  &\G_1(3)&\frac{27\eta(3\tau)^{12}}{\eta(\tau)^{12}+27\eta(3\tau)^{12}}&  -27\left(\frac{\eta(3\tau)}{\eta(\tau)}\right)^{12}&E_{2,3}(\tau)&(1-z_3)E_{2,3}(\tau)\\
 4&\G_1(2)&\frac{64\eta(2\tau)^{24}}{\eta(\tau)^{24}+64\eta(2\tau)^{24}}& -64\left(\frac{\eta(2\tau)}{\eta(\tau)}\right)^{24}&E_{2,2}(\tau)&(1-z_4)E_{2,2}(\tau)
\end{array}
$$  
  \caption{ The triangle groups $(\tfrac{d}{d-2},\infty,\infty)$ and relevant information}
    \label{table: Index2}
 \end{center}    
\end{table}

\renewcommand{\arraystretch}{1}

\subsection{Outline of the proofs}

In \S \ref{constructingMF}, we construct the spaces $\Gfam_{d,e}$ of weight-4 Hecke eigenforms arising from the partition \eqref{eq:partition} and the relevant modular background. To relate these families to our hypergeometric Galois representations, we rely crucially on \cite[Theorem 2.3]{HMM1}, which establishes, when $p\equiv1 \pmod{M(d,e)}$, a supercongruence modulo $p^2$ between ${\mathbf H}(\varkappa(d,e);1;\F_p)$ and
\[
F(\varkappa(d,e),1)_{p-1}+(-1)^{\tfrac{p-1}{d}+\tfrac{p-1}{e}}p.
\]
Since ${\mathbf H}(\varkappa(d,e);1;\F_p)$ takes values in $\Z$ here and its absolute value is bounded by $2p^{3/2}+p$, this supercongruence reduces the determination of its precise value to that of the truncated hypergeometric series modulo $p^2$. In \S \ref{ss:CFGL}, using CFGL, we relate $F(\varkappa(d,e),1)_{p-1}$ to the $p$th coefficient of any normalized Hecke eigenform $g_{d,e}\in\Gfam_{d,e}$ and, based on Proposition \ref{prop:key-modp}, determine the finite-order twist appearing in Theorem \ref{thm: Galois version}. Finally, in \S \ref{ss:modularity}, we assemble these ingredients to prove both of our main theorems.

\section{Constructing the \texorpdfstring{$\mathfrak G_{d,e}$}{G d,e} families  of modular forms}\label{constructingMF}

Now we construct the spaces $\Gfam_{d,e}$ of weight-4 modular forms on congruence subgroups of $\G_1(d)$, for $d=2$, $3$, $4$. Our construction is motivated by  \Cref{eq:Zagier}.

Combining the notations in Table \ref{table: Index2}, we define 
\begin{multline}\label{eq:g(d,e,i)}
    {\BG}_d(i/e):=\left(\tfrac{C_1(d)}4\right)^{-i/e}\left(\tfrac{z_d}{1-z_d}\right)^{i/e}F(\varkappa_2(d),z_d)^2\cdot \Theta z_d \\ 
    = \left(\tfrac 4{C_1(d)}\cdot \tfrac{z_d}{1-z_d}\right)^{i/e}(1-z_d)E^2_{2,6-d}, 
\end{multline}
where $\Theta=q\frac{d}{dq}$ 
and we fix our choice of the modular function $\left(\tfrac4{C_1(d)}\frac{z_d}{1-z_d}\right)^{1/e}$ that has leading $q$-coefficient one.

 For example, substituting $d = 3$ in \eqref{eq:g(d,e,i)} leads to the cusp forms 
\begin{equation}
\BG_{3}(i/e)(e\tau) = E_{2,3}(e \tau)^{2}  \cdot \frac{\eta(e \tau)^{12(1-i/e)}\eta(3e\tau)^{12(i/e)}}{\eta(e\tau)^{12}+27\eta(3e\tau)^{12}}.
\end{equation}
Now, to ensure congruence cusp forms, we restrict to the rational numbers $i/e \in (0,1)$ such that $12i/e$ is an integer. These eleven rational numbers are grouped into five families using the values $e \in \{2,3,4,6,12\}$. For example, consider the $e=3$ case. 
The corresponding cusp forms are 
$$
\BG_{3}(1/3)(3 \tau) =\frac{\eta(3\tau)^8\eta(9\tau)^4}{\eta(3\tau)^{12}+27\eta(9\tau)^{12}}E_{2,3}(3\tau)^2$$
and
$$\BG_{3}(2/3)(3 \tau) =\frac{\eta(3\tau)^4\eta(9\tau)^8}{\eta(3\tau)^{12}+27\eta(9\tau)^{12}}E_{2,3}(3\tau)^2
$$
on $\G_0(27)$. The corresponding $q$-expansions are
$$ \BG_{3}(1/3)(3 \tau) = q+q^{4}-25q^{7}+45q^{10}+20q^{13}-135q^{16} + \cdots$$
and
$$\BG_{3}(2/3)(3 \tau) = q^{2} + 5q^{5} - 7q^{8} - 5q^{11}-25q^{14} + 130q^{17} + \cdots.$$

From the hypergeometric perspective as described in  \cite{HMM1,grove2025hypergeometricmodularityconjecturesdawsey,rosen, rosenK1}, the subspace $\Gfam_{3,3}$ generated by $\{\BG_{3}(1/3)(3 \tau), \BG_{3}(2/3)(3 \tau) \}$ of the weight 4 cusp forms on $\G_0(27)$ is invariant under all the Hecke operators. To identify Hecke eigenforms $\beta_{3,3,1} \BG_{3}(1/3)(3 \tau)+ \beta_{3,3,2}\BG_{3}(2/3)(3 \tau)$ of $\Gfam_{3,3}$,  
it  suffices to consider the action of $T_2$ on $\Gfam_{3,3}$.  
In particular, we have 
$$
  T_2(\BG_{3}(1/3)(3 \tau))=9\BG_{3}(2/3)(3 \tau),  \qquad T_2(\BG_{3}(2/3)(3 \tau))=\BG_{3}(1/3)(3 \tau),
$$ 
using Sturm bounds. The relation $T_{2}^{2}=9I$ on $\Gfam_{3,3}$ allows us to choose $-3$ or $3$ for $\beta_{3,3,2}$ when taking $\beta_{3,3,1} = 1$. 
These give us the two normalized Hecke eigenforms 

$$f_{27.4.a.a}(\tau) = \BG_{3}(1/3)(3 \tau) -3  \BG_{3}(2/3)(3 \tau), $$
and 
$$f_{27.4.a.b}(\tau)= \BG_{3}(1/3)(3 \tau) + 3 \BG_{3}(2/3)(3 \tau).$$

We now give a more general result that works for all the cases.  
\begin{proposition}\label{lemma: constructingforms}
Let $d=2$, $3$,  $4$, and $M_d=8$, $12$, $24$, respectively. For each positive integer $i$ less than $M_d$,    $ {\BG}_d(i/M_d)(M_d\tau)$ is a congruence weight 4 cusp form. Its level and (primitive) Dirichlet character $\xi_{d,e}$ are $32$ with $\left(\frac 2\cdot \right)^{i}$, $432$ with $\left(\frac 3\cdot \right)^{i}$, or $1152$ with $\left(\frac{2}\cdot \right)^{i}$,  for $d=2$, $3$, and $4$, respectively.  

In particular, for each positive divisor $e>1$ of $M_d$, there are explicit 
$\beta_{d,e,i}$ for each $i\in (\Z/e\Z)^\times$ so that  $\beta_{d,e,i}^2\in\Z$ and 
           $$
            g_{d,e} (\tau)= \sum_{i\in (\Z/e\Z)^\times}\beta_{d,e,i} \cdot {\BG}_d(i/e)(e\tau)
           $$ is a Hecke eigenform determined up to tensoring with any Dirichlet character of conductor modulo $e$.  
\end{proposition}

\begin{proof}[Proof]
For a fixed $d \in \{2,3,4\}$, the function  
$$
   F(\varkappa_2(d),z_d)^2\cdot \Theta{z_d}=(1-z_d)E^2_{2,6-d}
$$
is a weight $4$ modular form on $\G_1(6-d)\simeq (\frac{d}{d-2}, \infty, \infty)$ with simple zeros at the elliptic point of order $\frac{d}{d-2}$ and the cusp $0$. The modular function  $\frac{z_d}{z_d-1}$   has a  simple zero at cusp $i\infty$, a simple pole at $0$, and value $1$ at the elliptic point of order  $\frac{d}{d-2}$. Hence the form ${\BG}_d(i/M_d)(\tau)$ is a modular form with a suitable character introduced by $\left(\frac{z_d}{1-z_d}\right)^{i/e}$ via the transformation formula of $\eta$-function, which ensures that ${\BG}_d(i/M_d)(M_d\tau)$ is a modular form on $\G_1(4\cdot 8)$, $\G_1(3\cdot 12^2)$, and $\G_1(2\cdot 24^2)$, for $d=2$, $3$, and $4$, respectively.

As in \cite{HMM1,grove2025hypergeometricmodularityconjecturesdawsey,rosen, rosenK1}, the subspace spanned  by the cusp forms 
$$
 {\BG}_d(i/M_d)(M_d\tau), \quad i=1, \ldots, M_d-1,
$$ over $\overline {\Q}$ 
are invariant under all the Hecke operators coprime to $de$. 

Now we can construct Hecke eigenforms using suitable combinations of ${\BG}_d(i/e)(e\tau)$. 
\end{proof}

\begin{remark}
\begin{enumerate}
\item  One can check the proof on a case-by-case basis as in \cite{HMM1,grove2025hypergeometricmodularityconjecturesdawsey}, or use a geometric argument as in \cite{rosen,rosenK1}. However, we omit the details, and assure the reader the details have been checked on a case-by-case basis as well. 
\item  
In  \Cref{EHMMappendix}, we give a chosen Hecke eigenform for each of the $\Gfam_{d,e}$ families in   \Cref{G2table,G3table,G4table} with the primitive character listed in the last column.  
\item  For $d=6$, the function $z_6$ such that $4z_6(1-z_6)=1728/j$ is not a modular function on any subgroup of $\SL_2(\Z)$, so we omit the relevant discussion here.  
\end{enumerate}
 
\end{remark}

\section{Commutative Formal Group Laws}\label{ss:CFGL}
\subsection{Background}
In the previous section, we have prepared the families of modular forms, as building blocks, to be used for our main theorems. 
We will first explain the nature of $ \varepsilon_{d,e}$  for  \Cref{thm:H-version}.  
This can be done via congruences. In this section, we will employ standard $p$-adic analysis. We will use techniques converting Gamma quotients to $p$-adic Gamma quotients explained in \cite{LR} combined with commutative formal group laws, as \S 3.6 of \cite{HMM1}. 

The next lemma will be used to deal with the coefficients of $\BG_d(i/e)$  and hence $g_{d,e}$ in $t$. It is an extension of \cite[Lemma 3.5]{HMM1}.
\begin{lemma}\label{lem:HG-product}
Set $\varkappa^\flat=\{\{a_1,\dots a_n\},\{b_1=1,\dots b_n\}\}$, $\varkappa^a=\{\alpha,\beta\}$. Assume   $1\in \beta$. As power series of $t$, 
      $$ F(\varkappa^a,t)F(\varkappa^\flat,t)=\sum_{k\ge 0} A(k)\cdot B(k)\cdot t^k$$  where 
    \begin{equation}\label{eq:A(k)}
     A(k)=\prod_{a\in \alpha, b \in \beta}\frac{(a)_k}{(b)_k}
\end{equation}
    and
\begin{equation}\label{eq:B(k)}
    B(k)= F(\varkappa^\flat\star\{\{1-b-k\}_{b\in \beta},\{1-a-k\}_{a\in \alpha}\},1),
\end{equation} which is a terminating series due to the non-positive integer $-k$ among $\{1-b-k\}_{b\in \beta}$.
\end{lemma}
\begin{proof}
   It follows from $(u)_{k-r}=(-1)^r\frac{(u)_k}{(1-u-k)_r},$ see \cite{Win3X}.
\end{proof}

We give two applications of this lemma. Recall that $\varkappa_2(d)=\{\{\tfrac{1}{d},\tfrac{d-1}{d}\},\{1,1\}\}$. 
\begin{example}\label{eg:1}  The product of the series $  (1-t)^{-1/2}F(\varkappa_2(2),t)$,  can be written as a power series in $t$ in two ways:
\begin{align*}
     (1-t)^{-1/2}&F(\varkappa_2(2),t)=\pFq10{\tfrac12}{}t F(\varkappa_2(2),t)\\
     =&\sum_{k\ge 0}\frac{(\tfrac12)_k}{k!}\cdot \pFq32{\tfrac12&-k& -k }{&\tfrac12-k&\tfrac12-k}1 t^k\\
     =&\sum_{k\ge 0}\frac{(\tfrac12)_k^2}{k!^2}\cdot \pFq32{\tfrac12&\tfrac12& -k }{&1&\tfrac12-k}1 t^k.
\end{align*}
\end{example}
\begin{example}\label{eg:2}
One presentation of the left-hand side of the Clausen formula \eqref{eq:Clausen} is   
 $$
   F(\varkappa_2(d),z)^2=\sum_{k\ge 0}\frac{(\tfrac1d)_k(\tfrac{d-1}d)_k}{k!^2}\cdot \pFq43{\tfrac1d&\tfrac{d-1}d&-k& -k }{&1&\tfrac1d-k&\tfrac{d-1}d-k}1 z^k. 
$$
\end{example}

We recall the necessary facts about commutative formal group laws (CFGL) below. This statement is a specialization of \cite[Theorem A.9]{Stienstra-Beukers}. 
For a fixed prime $p$, let $\Z_p$ be the ring of $p$-adic integers. 
\begin{theorem}[See \cite{Stienstra-Beukers}]\label{sb}
  Assume $R$ is a $\Z_p$-algebra and $\sigma$ an automorphism of $R$ so that $\sigma(x)\equiv x^p\pmod {pR}.$ Let $$\omega(u)=\sum_{n=1}^\infty c_nu^n\frac{du}u \quad \text{and} \quad u(q)=\sum_{n=1}^\infty b_n q^n, $$ 
  where $b_n,c_n\in R$ so that $b_1$ is invertible in $R$.  Further, write $$\omega(u(q))=\sum_{n=1}^\infty a_nq^n\frac{dq}q. $$ Then there exist $s_1,s_2\in R$ so that for all $m$ and $v$ positive integers, 
  $$
    a_{mp^v}-s_1\sigma(a_{mp^{v-1}})+s_2p\sigma^2(a_{mp^{v-2}})\equiv 0\pmod {p^v}
  $$ 
  if and only if $$c_{mp^v}-s_1\sigma(c_{mp^{v-1}})+s_2p\sigma^2(c_{mp^{v-2}})\equiv 0\pmod {p^v}.$$ 
\end{theorem}
Note that this statement is a slight correction of a typo in Proposition 3.4 of \cite{HMM2}, where $\sigma$ is applied to $\alpha_p$ and $\beta_p$ instead of the elements of the power series itself.  

\subsection{CFGL in the current setting}\label{ss:CFGLhere}
For the remainder of this section, we assume $d,e$ as \Cref{thm:H-version}  are fixed, $p\nmid de$ 
is a fixed prime, and 
\begin{equation}\label{eq:R}
  R=\Z_p[C_1(d)^{1/e}, \beta_{d,e,j}: j\in  (\Z/e\Z)^\times],  
\end{equation}where $\beta_{d,e,j}$ are the numbers as in Proposition  \ref{lemma: constructingforms}.  Let  $\sigma$ be an automorphism of $R$ such that for each $a\in R$, $\sigma(a)\equiv a^p\pmod{pR}.$    
Let $R^\times$ be the invertible elements in $R$. 
In our setting, the power series $\omega(t_d(q))$ is the Fourier series for the modular form $g_{d,e}$  in $R[[q]]$, while $\omega(u^e)$ is the $u$-expansion  
of the same series in $R[[u]]$. The Fourier coefficients of 
any Hecke eigenform satisfy the Hecke recursions for all primes $p\nmid de$  so that the above theorem is applicable. 

Recall that for $j\in(\Z/M(d,e)\Z)^\times$, \begin{equation}\label{eq:jHD}
    j\cdot \varkappa(d,e)=\{\{\tfrac{j}{d},1-\tfrac{j}{d},\tfrac{j}{e},1-\tfrac{j}{e}\},\{1,1,1,1\}\}=\{\{\tfrac{1}{d},\tfrac{d-1}{d},\tfrac{j}{e},\tfrac{e-j}{e}\},\{1,1,1,1\}\},
 \end{equation} is a conjugate datum of $\varkappa(d,e)$.

Below we apply the CFGL 
to relate $q$ and the $t_d$ coefficients of $g_{d,e}$.
\begin{proposition}\label{prop:key-modp} Assume $d\in\{2,3,4\}$ and $e\mid M_d$. Let $t_d$ be as in Table \ref{tab:bigger triangle}, $C_1(d)$ be its leading coefficient, $g_{d,e}=\sum_{i\in(\Z/e\Z)^\times}\beta_{d,e,i}\BG_d(i/e)$ be any normalized Hecke eigenform in $\Gfam_{d,e}$ as in Proposition \ref{lemma: constructingforms}.   Let $p\nmid de$   be a prime and $R$ as in \eqref{eq:R}.  
Then we have the following conclusions regarding the truncated series:
\begin{enumerate}
    \item for any prime $p\equiv j\pmod e$ where $j\in(\Z/e\Z)^\times$ such that $\beta_{d,e,j}\in R^\times$, 
\begin{equation}\label{eq:4F3-gde}
 \pFq43{\frac1e&\frac{e-1}e&\frac1d&\frac{d-1}d}{&1&1&1}{1}_{p-1}\equiv a_p({g_{d,e}})\cdot  \left({-C_1(d)/4}\right)^{(1-jp)/e} \cdot \beta_{d,e,j} \pmod {pR};
\end{equation}

\item  if $p\equiv 1\pmod {e}$, then for each $j\in(\Z/e\Z)^\times$,
\begin{equation}\label{eq:4F3-gde-2}
 \pFq43{\frac je&\frac{e-j}e&\frac1d&\frac{d-1}d}{&1&1&1}{1}_{p-1}\equiv 
 a_p({g_{d,e}})\cdot \left({-C_1(d)/4}\right)^{j(1-p)/e} \pmod {pR}.
\end{equation} 
\end{enumerate}

\end{proposition}
\begin{example} For $d=2$, $3$, or $4$,   and each prime $p>3$
$$
\pFq43{\frac12&\frac12&\frac1d&\frac{d-1}d}{&1&1&1}{1}_{p-1}\equiv  \left(\frac{-C_1(d)}p\right)a_p({g_{d,2}}) 
\pmod p. 
$$
\end{example}

\begin{remark}
\begin{enumerate}
\item Though \eqref{eq:4F3-gde} and \eqref{eq:4F3-gde-2} look similar,  they contain different information. While \eqref{eq:4F3-gde} applies to almost primes,  to be used in the proof of \Cref{thm: Galois  version}, \eqref{eq:4F3-gde-2} is about the compatibility among truncated series for conjugate data $j\cdot \varkappa(d,e)$, which will be needed for the proof of \Cref{thm:H-version}. 
It is coherent with the space $\Gfam_{d,e}$ being stable under each Hecke operator $T_p$ when $p\nmid de$.
    \item Let $\varepsilon_{d,e}$ be as in Theorem \ref{thm:H-version}. When $p\equiv 1\pmod {M(d,e)}$, 
    \begin{equation}\label{eq:chi(d,e)-congruence}
      \varepsilon_{d,e}(p)=\pm 1\equiv (-C_1(d)/4)^{(1-p)/e} \pmod{p}. 
   \end{equation} This can be checked case-wise. For example when $d=2$, $C_1(2)=-64$, $M_2=8$, so $(-C_1(2)/4)^{(1-p)/8}=16^{(1-p)/8}=2^{(1-p)/2}\equiv \left(\tfrac{2}{p}\right)\pmod{p}$. Similarly when $d=4$, $M_4=24$ and $C_1(4)=256$, so $(-C_1(4)/4)^{(1-p)/24}=(-2^6)^{(1-p)/24}=(-1)^{(1-p)/24}\cdot 2^{(1-p)/4}=\pm1 $ for each prime $p\equiv 1\pmod{24}$. The $d=3$ case can be checked similarly.   

\end{enumerate}

\end{remark}

The remainder of this section is devoted to the proof of Proposition \ref{prop:key-modp} and we continue to assume $p$ and $R$ are fixed as before.  First recall that the form $g_{d,e}$ is defined as a linear combination of $\BG_d$ functions as in Proposition \ref{lemma: constructingforms}.  
After choosing the principal branch of root, we fix as an analytic function that
\begin{equation}\label{gdei}\BG_d(j/e)\frac{dq}{q}=C_1(d)^{-j/e}t_d^{j/e}F(\varkappa_{cont}(j/e),t_d)F(\varkappa_3(d),t_d)\cdot \frac 1{t_d}\frac{qdt_d}{dq}.
\end{equation}

 We take $t=t_d$ throughout this section. From the series expansions for hypergeometric series and the modular forms $\BG_d(j/e)$ in \Cref{constructingMF}, we have that 
$C_1(d)^{j/e}\BG_d(j/e)\frac{dq}{q}  \in \Q[[t^{1/e}]] $ in terms of $t^{1/e}$-series 
and    in terms of  $q$-expansion, it has coefficients in $\Z$. 
As a series in $u=(t/ C_1(d))^{1/e}$, we will write 
$$ \BG_d(j/e)\frac{dq}{q}= e \sum_{n=0}^\infty \tilde c_{j,n}u^{ne+j} \frac{du}{u},\quad \tilde c_{j,n}\in\Z,$$
and
\begin{equation}
    g_{d,e}\frac{dq}q= e\sum_{j\in (\Z/e\Z)^\times}\sum_{n=0}^\infty \beta_{d,e,j}u^{ne+j} {\tilde c_{j,n}} \frac{du}{u}= \sum_{n=0}^\infty c_nu^n \frac{du}u.
\end{equation}The last expression means we call the re-ordered coefficients of $u^n$ to be $c_n$.
These two formal power series obtained from expressing $g_{d,e}$ in $q$ and $u$ respectively now satisfy the hypothesis of Theorem \ref{sb}, and so we have the following Lemma.

\begin{lemma}\label{lem:cfgl}
  Assume $j\in (\Z/e\Z)^\times$ and $\beta_{d,e,j}\in R^\times$, then

$$ c_{jp}\equiv \beta_{d,e,j}^{1-p}a_p(g_{d,e})\sigma(c_j) \pmod {pR}. $$
\end{lemma}

\begin{proof}  To apply Theorem \ref{sb}, 
take $\chi$ to be the Dirichlet character $\xi_{d,e}$ of $g_{d,e}$ as an analytic function as listed in Tables \ref{G2table}, \ref{G3table}, and \ref{G4table}. As $g_{d,e}$ is a weight-4 Hecke eigenform for all primes $p\nmid de$, the Hecke recursions imply that for all $m,v\ge 1$ \begin{equation}\label{recurs}
    a_{mp^v}-a_pa_{mp^{v-1}}+\chi(p)p^{3}a_{mp^{v-2}}\equiv 0\pmod {p^v R}.
\end{equation}
When $v=1$, it is reduced to 
\[ a_{mp}-a_pa_{m} \equiv 0 \pmod{pR}.\]
By Proposition \ref{lemma: constructingforms}, for $m=j$, $a_j=\beta_{d,e,j}\in R^\times$.   So 
$$ \sigma(a_j)=\sigma(\beta_{d,e,j})=\sigma(\beta_{d,e,j})a_j/\beta_{d,e,j}. $$ 
Thus the above congruence can be written as 
$$a_{jp}-a_p\beta_{d,e,j}/\sigma(\beta_{d,e,j})\cdot \sigma(a_j)\equiv a_{jp}-a_p(g_{d,e})\beta_{d,e,j}^{1-p}\cdot \sigma(a_j)\equiv 0\pmod{pR}.$$ By applying the CFGL 
isomorphism, it remains true if $a_n$ is replaced with $c_n$.
\end{proof}

Next we draw some conclusions relating to the terminating series of Lemma \ref{lem:HG-product}.   For a fixed prime $p$, let   $\G_p$ be the Morita $p$-adic Gamma function.  For each $a\in\Z_p$, we use $[a]_0$ to denote the first $p$-adic digit of $a$ and   $$a':=(a+[-a]_0)/p$$  the Dwork dash operation on $a$.
  \begin{lemma}\label{lem:mj} Let $e>1$ be a divisor of 24 and $p$ be a prime coprime to $e$. 
      \begin{enumerate}
          \item If $p\equiv j\pmod e$, then $
          [-1/e]_0=(pj-1)/e$.
          \item If $p\equiv 1\pmod e$, then for each $j\in(\Z/e\Z)^\times$, $
          [-j/e]_0=j(p-1)/e$.
      \end{enumerate} 
  \end{lemma}
 \begin{proof}
  (1) As $e\mid 24$, if $p\equiv j\mod e$, then $pj\equiv 1\pmod e$. As $1\le j<e$, $(pj-1)/e\in\Z$ is a positive integer less than $p$, and is congruent to $-1/e$ modulo $p$. Thus $[-1/e]_0=(pj-1)/e$.

  (2) When $p\equiv 1\pmod e$, $j(p-1)/e$ is also a positive integer less than $p$, which is congruent to $-j/e$. Thus $[-j/e]_0=j(p-1)/e$.
 \end{proof} 

To use Lemma \ref{lem:HG-product}, $\varkappa^\flat=\{\{\tfrac12,\tfrac{1}{d},\tfrac{d-1}{d}\},\{1,1,1\}\}$ and $\varkappa^a=\varkappa_{alg}(\tfrac le)=\{\{\tfrac{l}e,\tfrac{l}e+\tfrac{1}{2}\}\},\{\tfrac{2l}e,1\}\}$ so
\begin{equation}\label{eq:Aj}
A_l(k)=\frac{\left(\frac{l}{e}\right)_k\left(\frac{l}{e}+\frac12\right)_k}
{\left(\frac{2l}{e}\right)_k(1)_k}
=\frac{\left(\frac{2l}{e}\right)_{2k}2^{-2k}}
{\left(\frac{2l}{e}\right)_k(1)_k},\end{equation} where the last claim is due to the duplication formula $(a)_k(a+\tfrac{1}{2})_k=(2a)_{2k}2^{-2k}$; and 
\begin{equation}\label{eq:Bj}
B_l(k)=\pFq54{\tfrac12&\tfrac{1}{d}&\tfrac{d-1}{d}&-k&1-\tfrac{2l}e-k}{&1&1&1-\tfrac{l}e-k&1-\tfrac{l}e-\tfrac{1}{2}-k}1.
\end{equation} 
 
\begin{lemma}\label{lem:bip}
Let $p\nmid de$ be a prime and $A_{j}(k)$ as in \eqref{eq:Aj}.   
    \begin{enumerate} 
   \item For $p\equiv j\pmod { e}$, we have  $$c_{jp}\equiv { C_1(d)^{-1/e}}A_1\left(\frac{jp-1}{e}\right)\pfq{4}{3}{\tfrac1e&\tfrac{e-1}e&\tfrac1d&\tfrac{d-1}d}{&1&1&1}{1}_{p-1}\pmod p.$$ 

   \item For $p\equiv 1 \pmod e$ and $j\in(\Z/e\Z)^\times$ such that $\beta_{d,e,j}\in R^\times$, then $$c_{jp}\equiv { \beta_{d,e,j}}{ C_1(d)^{-j/e}}A_j\left(\frac{pj-j}{e}\right)\pfq{4}{3}{\tfrac{j}e&\tfrac{e-j}e&\tfrac1d&\tfrac{d-1}d}{&1&1&1}{1}_{p-1}\pmod p.$$
     \end{enumerate}
\end{lemma}
\begin{proof}
When $p\equiv j \pmod e$, $jp \equiv 1 \pmod e$ as $e\mid 24$, hence  $c_{jp}$ comes from the $ u^{1+jp-1}= t^{\tfrac 1e+\tfrac{jp-1}{e}}$  coefficient of $g_{d,e,1}$. Thus the appearance of the $C_1(d)^{-1/e}$ is obvious in light of \eqref{gdei}.  Applying Lemma \ref{lem:HG-product} to the product of hypergeometric series with chosen $\varkappa^\flat$ and $\varkappa^a$, we obtain the $\tfrac{jp-1}{e}$th coefficient of $t$. 

For the terminating series $B_1(\tfrac{jp-1}{e})$, we have 
\begin{multline*} 
    \pFq54{\frac12&\frac1d&\frac{d-1}d&-\tfrac{jp-1}{e}&1-\frac 2e-\tfrac{jp-1}{e}}{&1&1&1-\tfrac{jp-1}{e}-\frac1e&\frac12-\tfrac{jp-1}{e}-\frac1e}{1}\\
    =\pFq54{\frac12&\frac1d&\frac{d-1}d&\frac{1-jp}e&\frac {e-1}e+\frac{jp}e}{&1&1&1-\frac{jp}e&\frac12-\frac{jp}e}{1}_{\tfrac{jp-1}{e}} \equiv 
    \pFq54{\frac1d&\frac{d-1}d&\frac{1}{e}&\frac {e-1}e&\frac12}{&1&\frac12&1&1}{1}_{\tfrac{jp-1}{e}}
   \\= \pFq43{\frac1d&\frac{d-1}d&\frac{1}{e}&\frac {e-1}e}{&1&1&1}{1}_{\tfrac{jp-1}{e}} \equiv 
    \pFq43{\frac1d&\frac{d-1}d&\frac{1}{e}&\frac {e-1}e}{&1&1&1}{1}_{p-1} \pmod p.
\end{multline*}
The last congruence is followed by the fact that 
$$
\frac{(j/M)_k(1-j/M)_k}{k!^2}\equiv 0 \pmod p, \quad \mbox{when }   [-j/M]_0<k\le p-1,
$$
for any integer $M$ coprime to $p$ and $0<j/M<1$.

Similarly, we can derive the second conclusion by looking at $c_{jp}$ comes from $\BG_{d,e}(\tfrac{j}e)$  times the corresponding scalar multiple $\beta_{d,e,j}$.
\end{proof}

We now deal with the congruence property of the $A_j$ term.
\begin{lemma}\label{lem:A(k)} 
    For  each $j\in (\Z/ e\Z)^\times$,  
let $m_j=[-j/e]_0$. 
Then we have:
\begin{enumerate}    
\item   $$A_1\!\left(m_1\right)\equiv (-1/4)^{m_1}  \pmod p.$$
\item If $p\equiv 1\pmod e$, then for each   $j\in (\Z/ e\Z)^\times$,
$$A_j\!\left(m_j\right)\equiv (-1/4)^{m_j}  \pmod p.$$
\end{enumerate}
\end{lemma}
\begin{proof}

We first assume that $p\equiv j\pmod e$  so $m_1=(pj-1)/e$ as in  Lemma \ref{lem:mj}. Thus $\tfrac{2}e+2m_1=\tfrac{2pj}e$.  Using $(a)_k=\frac{\G(a+k)}{\G(a)}$,  \eqref{eq:Aj} can be written as
\begin{equation}\label{eq:Aj(mj)}
A_1\!\left(m_1\right)
=
\frac{
\Gamma\!\left(\frac{2pj}{e}\right)\Gamma\!\left(\frac{2}{e}\right)\Gamma(1)
}{
\Gamma\!\left(\frac{2}{e}\right)
\Gamma\!\left(\frac{2}{e}+m_1\right)
\Gamma\!\left(1+m_1\right)
}
\,2^{-2m_1}.
\end{equation}

To continue, we use the technique in \cite{LR} by Long--Ramakrishna to convert the  Gamma quotients to $p$-adic Gamma quotients from which modulo conclusions can be drawn more systematically. Note that some terms below may need an additional $p$-power  due to $\frac 2e+[ -\frac 2e]_0=p(\frac 2e)'\in p\Z_p$, where $'$ denotes Dwork dash operation which is defined by $a'=(a+[-a]_0)/p$ for each $a \in \Z_p$. This will be only needed when $[ -\frac2e]_0<\frac{jp-1}e$, which corresponds to $\delta_{j>1}$, which is 1 when $j>1$, and else it is 0. If $j=1$, then $[ -\frac2e]_0=\frac{2p-2}e>\frac{p-1}e.$ When $j\ge 2$, $[-\frac2e]_0\le\frac{2p(e-1)-2-ep}e<\frac{p(e-1)-1}e$.

Using Lemma 17 of  \cite{LR}, 
\[
\frac{\Gamma\!\left(\frac{2jp}{e}\right)}{\Gamma\!\left(\frac{2}{e}\right)}
=
(-1)^{2m_1}  \left[\left(\frac 2e\right)'p \right]^{\delta_{j>1}}
\frac{\Gamma_p\!\left(\frac{2jp}{e}\right)}{\Gamma_p\!\left(\frac{2}{e}\right)},
\]

\[
\frac{\Gamma\!\left(\frac{2}{e}\right)}
{\Gamma\!\left(\frac{2}{e}+\frac{jp-1}{e}\right)}
=
(-1)^{m_1}  
\left[\left(\frac 2e\right)'p \right]^{-\delta_{j>1}}
\frac{\Gamma_p\!\left(\frac{2}{e}\right)}
{\Gamma_p\!\left(\frac{2}{e}+\frac{jp-1}{e}\right)},
\]
and
\[
\frac{\Gamma(1)}{\Gamma\!\left(1+\frac{jp-1}{e}\right)}
=
(-1)^{m_1}
\frac{\Gamma_p(1)}{\Gamma_p\!\left(1+\frac{jp-1}{e}\right)}.
\] 
 
Hence
\begin{align*}
A\!\left(m_1\right)
&=
\frac{
\Gamma_p\!\left(\frac{2jp}{e}\right)\Gamma_p\!\left(\frac{2}{e}\right)\Gamma_p(1)
}{
\Gamma_p\!\left(\frac{2}{e}\right)
\Gamma_p\!\left(\frac{2}{e}+\frac{jp-1}{e}\right)
\Gamma_p\!\left(1+\frac{jp-1}{e}\right)
}
\,2^{-2m_1}\\&=-2^{-2m_1}\frac{\Gamma_p\!\left(\frac{2jp}{e}\right)}{\Gamma_p\!\left(\frac{2}{e}+\frac{jp-1}{e}\right)
\Gamma_p\!\left(1+\frac{jp-1}{e}\right) },
\end{align*} 
and by the reflection formula of  $p$-adic Gamma functions, 

\begin{equation*}
A\!\left(m_1\right)\equiv -2^{-2m_1}\frac{\Gamma_p\!\left(0\right)}{\Gamma_p\!\left(\frac{1}{e}\right)\Gamma_p\!\left(1-\frac{1}{e}\right) }\equiv -2^{-2m_1} (-1)^{[1/e]_0}
\equiv (-1/4)^{\frac{jp-1}{e}}  \pmod p.
\end{equation*}
We now handle the  case when $p\equiv 1\pmod e$ and $j$ is any element in $(\Z/e\Z)^\times$. By Lemma \ref{lem:mj}, $m_j=j(p-1)/e$ and $\tfrac{2jp}{e}-(\tfrac{jp}{e}+m_j)=m_j\in\Z$. We simplify \eqref{eq:Aj(mj)} as
\[
    A_j\!\left(m_j\right)
=\frac{
\Gamma\!\left(\frac{2pj}{e}\right)\Gamma\!\left(\frac{2j}{e}\right)\Gamma(1)
}{
\Gamma\!\left(\frac{2j}{e}\right)
\Gamma\!\left(\frac{2j}{e}+m_j\right)
\Gamma\!\left(1+m_j\right)
}
\,2^{-2m_j}=
\frac{
\Gamma\!\left(\frac{2jp}{e}\right)\Gamma(1)
}{
\Gamma\!\left(\frac{2j}{e}+m_j\right)
\Gamma\!\left(1+m_j\right)
}.\\
\,2^{-2m_j}. \]
As there is no $p$ multiple either within the set $\{\tfrac{2j}{e}+m_j, \tfrac{2j}{e}+m_j+1,\cdots,\tfrac{2jp}e-1\}$ nor $\{1, 2,\cdots,  m_j\}$ of the same cardinality,
\[
    A_j\!\left(m_j\right)
=\frac{
\Gamma_p\!\left(\frac{2jp}{e}\right)\Gamma_p(1)
}{
\Gamma_p\!\left(\frac{2j}{e}+m_j\right)
\Gamma_p\!\left(1+m_j\right)
}
\,2^{-2m_j}
\equiv (-1/4)^{m_j} \pmod p. 
\]
This completes the proof. 
\end{proof}

Finally, we put all of the pieces together.

\begin{proof}[Proof of Proposition \ref{prop:key-modp}] 
We first assume $p\equiv j\pmod e$. By definition, $c_j$ is the $j$th coefficient in the $t$-expansion of $g_{d,e}$, and thus the first coefficient of $\beta_{d,e,j}\cdot \BG_{d,e}(\tfrac{j}e)$.  which is $C_1(d)^{-j/e}$. 
Hence $\sigma(c_j)=\sigma(\beta_{d,e,j})C_1(d)^{-jp/e} \pmod{pR}$. 
By Lemma \ref{lem:cfgl}, $$c_{jp}\equiv a_p\beta_{d,e,j}^{1-p}{ C_1(d)^{-jp/e}}\pmod {pR}.$$ 
We now compute $c_{jp}$ in two different ways.
 \begin{align*}   c_{jp}&\overset{\text{Lemma \ref{lem:cfgl}}} \equiv  \beta_{d,e,j}^{1-p}a_p(g_{d,e}) C_1(d)^{-jp/e}\\& \overset{\text{Lemma \ref{lem:bip}}}\equiv C_1(d)^{-1/e}A_1(m_1) \pFq43{\tfrac1e&1-\tfrac1e&\tfrac1d&1-\tfrac1d}{&1&1&1}{1}_{p-1} \pmod {p}\\& \overset{\text{Lemma \ref{lem:A(k)}}}\equiv C_1(d)^{-1/e}(-1/4)^{(jp-1)/e} \pFq43{\tfrac1e&1-\tfrac1e&\tfrac1d&1-\tfrac1d}{&1&1&1}{1}_{p-1} \pmod {p}.
 \end{align*}    
 Putting together and dividing by $C_1(d)^{-1/e}(-1/4)^{(jp-1)/e}$ we reach the claim.

Similarly, when $p\equiv 1\pmod e$, 
both $c_{jp}$ and $c_j$ are from $\BG_{d,e}(\tfrac{j}e)$. Thus $c_j=C_1(d)^{j/e}\beta_{d,e,j}$ and $\sigma(c_j)\equiv C_1(d)^{jp/e}\beta_{d,e,j}^p\pmod {pR}$. By   Lemma \ref{lem:cfgl}, \[
     c_{jp}\equiv \beta_{d,e,j}^{1-p}a_p(g_{d,e})\sigma(c_j) \equiv \beta_{d,e,j}^{} a_p(g_{d,e})C_1(d)^{jp/e}\pmod{pR}.
\]
By  Lemmas \ref{lem:bip} and \ref{lem:A(k)}, 
\begin{eqnarray*}  
c_{jp} &\equiv& \beta_{d,e,j}C_1(d)^{- j/e}A_j(m_j)\pfq{4}{3}{\tfrac{j}e&\tfrac{e-j}e&\tfrac1d&\tfrac{d-1}d}{&1&1&1}{1}_{p-1}\\
&\equiv& \beta_{d,e,j} C_1(d)^{-j/e} (-1/4)^{m_j}\pfq{4}{3}{\tfrac{j}e&\tfrac{e-j}e&\tfrac1d&\tfrac{d-1}d}{&1&1&1}{1}_{p-1} \pmod{pR}.
\end{eqnarray*}
 Equating both yields the desired claim.
\end{proof}

\section{Modularity}\label{ss:modularity}
\subsection{Hypergeometric Galois representations}
In this section, we prove the second half of \Cref{thm:H-version} and \Cref{thm: Galois version}.  As an immediate corollary of \Cref{prop:key-modp}, we can certify the Hecke eigenform providing modularity of the Galois representation $\eta_{\varkappa(d,e),\ell,1}$ in  \Cref{thm: Galois version}. Note that we assume $\eta_{\varkappa(d,e),\ell,1}$ is semisimple. Let $\Frob_{\wp}$ denote the geometric Frobenius element.

\begin{proposition}[\cite{BCM,KatzESDE}]\label{prop:galois-reps}
    For each $d,e$ as Theorem \ref{thm:H-version}, there is a 3-dimensional Galois representation $\eta_{\varkappa(d,e),\ell,1}$ of $G_{\Q(\zeta_M)}$ such that for each nonzero prime ideal $\wp$ of $\Z[\zeta_M]$ coprime to $de$
    $$\Tr\, \eta_{\varkappa(d,e),\ell,1}(\Frob_\wp)={\mathbf H}(\varkappa(d,e);1;\wp)$$ so that, 
\begin{equation}\label{eq:eta-decomp}
\eta_{\varkappa(d,e),\ell,1}
\cong
\rho_{d,e,\ell}\oplus
\left.\varphi_{d,e}\epsilon_\ell\right|_{G_{\Q(\zeta_M)}},
\end{equation}
where $\rho_{d,e,\ell}$ is two-dimensional, $\epsilon_\ell$ is the cyclotomic character, and $\varphi_{d,e}$ is a finite-order character. Moreover, the two eigenvalues of $\rho_{d,e,\ell}(\Frob_\wp)$ have the same absolute value
$
\left|\Z[\zeta_M]/\wp\right|^{3/2}$
under every archimedean embedding.  

If $\varkappa(d,e)$ is defined over $\Q$, then $\eta_{\varkappa(d,e),\ell,1}$ has an extension, denoted by $\eta_{\varkappa(d,e),\ell,1}^{BCM}$, to $G_\Q$ such that for each $p\nmid de$
$$\Tr \, \eta_{\varkappa(d,e),\ell,1}^{BCM}(\Frob_{\wp})={\mathbf H}(\varkappa(d,e);1;\F_p),$$ { in which the formula given in \cite[Theorem 1.3]{BCM} is used for ${\mathbf H}(\varkappa(d,e);1;\F_p)$.}
\end{proposition}
\begin{proof}
    The proposition follows from the work of Katz \cite{KatzESDE} for the first part and Beukers--Cohen--Mellit \cite{BCM} for the cases defined over $\Q$.  See also Theorem 4 of \cite{LTYZ}.
\end{proof}

To determine the finite-order character  $\varphi_{d,e}$, we apply Theorem~2.3 of \cite{HMM1}, which is applicable because the $\beta$ set of $\varkappa(d,e)$ consists entirely of $1$.  In this case, the sign of the extra term can be computed using the reflection formula for $\G_p(\cdot)$:
\begin{equation}\label{eq:p-adic-sign}
\frac{\Gamma_p({\beta})}{\Gamma_p({\alpha})}
=
\frac{1}{\Gamma_p(1/d)\Gamma_p(1-1/d)\Gamma_p(1/e)\Gamma_p(1-1/e)}
=
(-1)^{[-1/d]_0}(-1)^{[-1/e]_0},
\end{equation}
where, as before, $[\cdot]_0$ denotes the first $p$-adic digit.  Namely, if $\wp$ lies above a prime $p$ that splits completely in $\Q(\zeta_M)$, then
$$
\varphi_{d,e}(\Frob_\wp)
=
(-1)^{[-1/d]_0}(-1)^{[-1/e]_0}.
$$
For such a splitting prime $p$, the Gross--Koblitz formula 
gives an embedding
$$
{\mathbf H}(\varkappa(d,e);1;\wp)\hookrightarrow \Z_p,
$$ for which we adopt the choice used in \cite{HMM1}.
Consequently, given that ${\mathbf H}(\varkappa(d,e);1;\wp)\in\Z_p^\times$, and because of \eqref{eq:eta-decomp}, exactly one of the the three eigenvalues for $\eta_{\varkappa(d,e),\ell,1}(\Frob_\wp)$ embedded into $\Z_p$ is a $p$-adic unit. We denote this unit eigenvalue by $\mu_{\varkappa(d,e),p,1}$.

For $c\in\Z/(M(d,e)\Z)^\times$,  recall that $$c\cdot \varkappa(d,e):=\{\{\tfrac{c}{d},1-\tfrac{c}{d},\tfrac{c}{e},1-\tfrac{c}{e}\},\{1,1,1,1\}\}=\{\{\tfrac{1}{d},1-\tfrac{1}{d},\tfrac{c}{e},1-\tfrac{c}{e}\},\{1,1,1,1\}\}. $$

The conclusions established for $\eta_{\varkappa(d,e),\ell,1}$ remain valid for $\eta_{c\cdot \varkappa(d,e),\ell,1}$.

\begin{lemma}\label{lem:thm1-step1} 
    We assume notation and assumptions as in Theorem \ref{thm:H-version}. Further take $\mu_{\varkappa(d,e),p,1}$ to be the $p$-adic unit root associated to $\varkappa(d,e)$. Then for each $c\in (\Z/M(d,e)\Z)^\times$ and prime $p\equiv 1 \pmod{M(d,e)}$
    \begin{align}\label{eq:CY3-modp3}
    \begin{split}
    {\mathbf H}(c\cdot \varkappa(d,e),1,\F_p)-(-1)^{(p-1)/d+(p-1)/e}p &\equiv F(c\cdot \varkappa(d,e),1)_{p-1} \\&\equiv \mu_{c\cdot \varkappa(d,e),p,1}\pmod {p^2}.
    \end{split}
\end{align}
\end{lemma}
\begin{proof}
The claim follows from Theorem 2.3 of \cite{HMM1}, the preceding discussion, and the vital role played by the Dwork unit root $ \mu_{c \cdot \varkappa(d,e),p,1}$ associated with the truncated series $F(c\cdot \varkappa(d,e),1)_{p^s-1}$. In particular, the cited Theorem  gives the following two supercongruence conclusions:
    \begin{equation}
        {\mathbf H}(c\cdot \varkappa(d,e),1,\F_p)-(-1)^{(p-1)/d+(p-1)/e}p\equiv \mu_{c \cdot \varkappa(d,e),p,1}\mod {p^2}
    \end{equation}
in which the sign, computed using the reflection formula, is independent of $c$, and
    \begin{equation}
       \mu_{c \cdot\varkappa(d,e),p,1}\equiv F(c\cdot \varkappa(d,e),1)_{p-1}\mod{p^2}.
    \end{equation}

These are equivalent to the claims of the Lemma.  
\end{proof} 
\subsection{Proof of Theorem \ref{thm:H-version}}

\begin{proof}[Proof of Theorem \ref{thm:H-version}]
First, we prove that the claim holds modulo $p^2$ for all conjugates, that is for all $p\equiv 1 \pmod{ M(d,e)}$ \begin{equation}\label{eq:modp^2}
\BH\left(
c\cdot \varkappa(d,e);1;\wp
\right)
\equiv 
\varepsilon_{d,e}(p)  a_p(g_{d,e})
+
(-1)^{\frac{p-1}{d}+\frac{p-1}{e}}p\pmod{p^2},
\end{equation} for $c\in (\Z/M\Z)^\times$. The bridge between these two is the truncated hypergeometric series and Lemma \ref{lem:thm1-step1}. 
By the second claim of  of Proposition \ref{prop:key-modp}, 
\begin{equation}\label{eq:modp}F(c\cdot \varkappa(d,e),1)_{p-1}\equiv \varepsilon_{d,e}(p)\cdot a_p(g_{d,e})\pmod{pR},\end{equation} and since $\varepsilon_{d,e}(p)=\pm 1$, this congruence can be considered mod $p$. Let   $\mu_{d,e,p}$ be the root of $T^2-a_p(g_{d,e})T+\xi_{d,e}(p)p^3$ in $\Z_p^\times$ where $\xi_{d,e}$ is the character of the modular form described in  Lemma \ref{lemma: constructingforms}.    The isomorphism of formal group laws used to prove \eqref{eq:modp} implies that the unit roots $\mu_{c\cdot \varkappa(d,e),p,1}$ and $\varepsilon_{d,e}(p)\mu_{d,e,p}$ are equal for each $c\in (\Z/M(d,e)\Z)^\times$ and $p\equiv 1\mod M$ so that $a_p(g_{d,e})\not\equiv 0\mod p$. Meanwhile, by the theory of normalized Hecke eigenforms,  
{$$
\mu_{c\cdot \varkappa(d,e),p,1} \equiv \varepsilon_{d,e}(p)\mu_{d,e,p}+\varepsilon_{d,e}(p)\xi_{d,e,p}(p)p^3/\mu_{d,e,p}\equiv \varepsilon_{d,e}(p)a_p(g_{d,e})   \pmod{p^3} 
$$}
As result, Lemma \ref{lem:thm1-step1} combined with \eqref{eq:modp} implies \eqref{eq:modp^2}. 

Following the method of \cite{HMM1}, we prove that this congruence actually is an equality. Work of Katz \cite{KatzESDE} (see  \cite[Theorem 4.1]{HMM1} for this result specialized to our case) implies $$|\BH\left(
c\cdot \varkappa(d,e);1;\wp
\right)|\leq 3p^{3/2}.$$ Likewise, for modular forms, we have the Weil bound, which states $|\varepsilon_{d,e}(p)a_p(g_{d,e})|\leq 2p^{3/2}$. As a result, we must have $$|\BH\left(
c\cdot \varkappa(d,e);1;\wp
\right)|-|\varepsilon_{d,e}(p)a_p(g_{d,e})|\leq 5p^{3/2},$$ yet the left-hand side is also equivalent to zero mod $p^2$. For primes $p>23$, we have $p^2>5p^{3/2}$, and so this has to be an equality. The conclusion for small primes can be verified by direct computation. When $\BH\left(
c\cdot \varkappa(d,e);1;\wp
\right)$ is an integer, this yields the desired result. Hence, when $\varkappa(d,e)$ is defined over $\Q$, the proof is complete. Otherwise, we now must show \begin{equation}
    \BH\left(
c\cdot \varkappa(d,e);1;\wp
\right) \in \Z.
\end{equation} Note that by \eqref{eq:Sta-HD}, $St_{c\cdot \varkappa(d,e)}=\{i\mid i\in (\Z/e\Z)^\times, i\equiv \pm 1\pmod e\}$ as $d\in\{2,3,4\}$. Thus the datum $c\cdot \varkappa(d,e)$ is defined over the totally real field $\Q(\zeta_e+\zeta_e^{-1})$ which is a subfield of $\Q(\sqrt{2},\sqrt{3})$ as $e\mid 24$. 

This implies that $\BH\left(c\cdot \varkappa(d,e);1;\wp
\right)$ lies in the ring of integers
of $\Q(\sqrt{2},\sqrt{3})$. 
As a result, $$\nu_1:=\BH\left(
 \varkappa(d,e);1;\wp
\right)-
(-1)^{\frac{p-1}{d}+\frac{p-1}{e}}p=u+v\sqrt{3}+\frac{w\sqrt{2}+t\sqrt{6}}{2},$$ for $u,v,w,t\in \Z$, with $w\equiv t\mod 2$. Note $\BH\left(
 \varkappa(d,e);1;\wp
\right)$ is a Galois conjugate of $\BH\left(c\cdot
 \varkappa(d,e);1;\wp
\right)$ for $c\in (\Z/M\Z)^\times$, again by straightforward properties of Gauss sums, and so multiplying by $c$ corresponds to switching the signs of $v,w,t$. Let $\nu_c$ denote the conjugate of $\nu_1$ given by multiplication by $c$. Hence, $$u=\frac{1}{\varphi(M)}\sum_{c\in (\Z/M\Z)^\times}\nu _c\equiv \varepsilon_{d,e}(p)  a_p(g_{d,e})\pmod{p^2}.$$ However, we know that individually, each $\nu_c$ is congruent to $\varepsilon_{d,e}(p)  a_p(g_{d,e})$ from \eqref{eq:modp^2}, and so the only possibility is that $p^2$ divides each of $v,w,t$. Using the same bounds as above, this is only possible if $v,w,t=0$, and the claim follows.
\end{proof}

\subsection{Proof of Theorem \ref{thm: Galois version}}

\begin{proof}[Proof of Theorem \ref{thm: Galois version}]

If $e\in \{2,3,4,6\}$, then $\varkappa(d,e)$ is defined over $\Q$. In this case, the $G_\Q$ representation $\eta_{\varkappa(d,e),\ell,1}^{BCM}$ from Proposition \ref{prop:galois-reps}  also admits a decomposition of the form
\begin{equation}\label{eq:gal-decomp}
\eta_{\varkappa(d,e),\ell,1}^{BCM}\cong \rho_{f^\sharp_{d,e}}\oplus \varsigma(d)\otimes\varsigma(e)\otimes\epsilon_\ell    
\end{equation}
extending \eqref{eq:eta-decomp}, with the modular form up to a twist determined by Theorem \ref{thm:H-version}; see \cite{LTYZ}. Among them $\rho_{f^\sharp_{d,e}}$ is a 2-dimensional representation of $G_\Q$ with Frobenius traces in $\Z$. 
When $\rho_{f^\sharp_{d,e}}$ admits CM, namely, when $(d,e)=(3,4)$ or $(4,3)$, the two cases correspond to the same representation. It therefore suffices to determine $\rho_{f^\sharp_{4,3}}$. As $g_{4,3}=f_{9.4.a.a}=\eta(3\tau)^8$ which admits CM by $\Q(\sqrt{-3})$, and by the second claim of  Proposition \ref{prop:key-modp},  $$(-C_1(4)/4)^{(p-1)/3}=(-64)^{(p-1)/3}\equiv 1\pmod p \quad \text{for} \quad p\equiv 1\pmod 3.$$ Thus $\rho_{f^\sharp_{4,3}}$ is isomorphic to $\rho_{f_{9.4.a.a}}$. 

Next we deal with the non-CM cases.   Theorem \ref{thm:H-version}, together with the Chebotarev density theorem, implies that $\rho_{f^\sharp_{d,e}}|G_{\Q(\zeta_e)}$ and $\rho_{g_{d,e}}|G_{\Q(\zeta_e)}$ are isomorphic. The extensions of these representations to $G_\Q$, $\rho_{f^\sharp_{d,e}}$ and $\rho_{g_{d,e}}$ are isomorphic up to a finite order character, which must be order at most 2 since both $\rho_{g_{d,e}}$ and $\rho_{f^\sharp_{d,e}}$ have Frobenius traces in $\Z$. It therefore remains only to determine this finite-order character. By the first claim of Proposition \ref{prop:key-modp}, it is determined by $(-C_1(d)/4)^{(1-jp/e)}\beta_{d,e,j}\pmod p$.\footnote{It is unclear otherwise why  $(-C_1(d)/4)^{(1-jp/e)}\beta_{d,e,j} \equiv \pm 1 \pmod p$.} Explicit calculations show that it is the trivial character when $d=2$ and $\chi_{-d}$ when $d=3,4$, as stated in Theorem \ref{thm: Galois version}.
\end{proof}
\begin{remark}
    The isomorphism of $G_\Q$ representations can also be obtained directly using supercongruences in \cite{LLL26,LTYZ} combined with Proposition \ref{prop:key-modp} via the same argument as above without appealing to \eqref{eq:gal-decomp}. We omit the details.
\end{remark}

\section{Future problems}\label{sec:future}

As discussed in the literature by e.g. \cite{HMM2, LLT2, Zagier-top-diff,Zudilin-SIGMA}, we expect a link between the hypergeometric evaluations to the periods of modular forms.  For example, in \cite{Zagier-top-diff}, Zagier shows that 
$$
\frac{\pi^2}{16} \pFq43{\tfrac12&\tfrac1{2}&\tfrac12&\tfrac 12}{&1&1&1} 1= L(f_{8.4.a.a},1). 
$$
When $d=2$, $3$, $4$, via  the modular forms construction in \Cref{constructingMF}, we have 
\begin{align*}
F( \varkappa(d,2),1)=& \frac1{2\pi i} \oint_{|t_d|=1} F(\varkappa_{alg}(1/2),1/t_d)F(\varkappa_3(d),t_d)\frac{dt_d}{t_d}\\
=& \frac1{2\pi i} \oint_{|4z_d(1-z_d)|=1} \left( \frac{4 z_d}{1-z_d}\right)^{1/2}F(\varkappa_2(d),z_d))^2\frac{dz_d}{z_d}.
\end{align*}
Similarly, we have the following expression in terms of the Euler integral formula, 
$$
F( \varkappa(d,2),1)=  \frac 2{\pi}\int_0^{1/2}\left( \frac{ z_d}{1-z_d}\right)^{1/2}F(\varkappa_2(d),z_d))^2\frac{dz_d}{z_d}. 
$$
\begin{example} For $d=2$, we have
    $$
\frac{\pi}4 \pFq43{\tfrac12&\tfrac1{2}&\tfrac12&\tfrac 12}{&1&1&1} 1= L(f_{16.4.a.a},1)
$$
given in \cite{RWZ}.

Applying a similar technique in \cite{RWZ},  we have the identity 
$$
3\int_0^{1/2}\left( \frac z{1-z}\right)^{1/2}\pFq21{\tfrac1{2}&\tfrac{1}{2}}{&1}{z}^2\frac{dz}{z}=\int_{1/2}^1\left( \frac z{1-z}\right)^{1/2}\pFq21{\tfrac1{2}&\tfrac{1}{2}}{&1}{z}^2\frac{dz}{z},
$$
which implies that 
\begin{align*}
   \pi  F( \varkappa(d,2),1)
   &= \frac{1}2 \int_0^1\left( \frac {z_2}{1-z_2}\right)^{1/2}\pFq21{\tfrac1{2}&\tfrac{1}{2}}{&1}{z_2}^2\frac{dz_2}{z_2}\\
     &=  \frac{1}2  \int^{\frac12}_{\infty} \sqrt{-16}\cdot 2\pi i \cdot f_{8.4.a.a}(\tau/2)d\tau = 8\pi   \int_{\frac14}^{\infty} f_{8.4.a.a}(\tau)d\tau\\
      &= 8\pi   \int_{0}^{\infty} f_{8.4.a.a}(\tau+1/4)d\tau. 
\end{align*}
Note that $ f_{8.4.a.a}(\tau+1/4)=f_{16.4.a.a}(\tau)$ (this can be also deduced from the transformation formula of $\eta(\tau+1/2)$). Hence we have 
$$
 \pi  F( \varkappa(d,2),1)=4 \cdot L(f_{16.4.a.a},1). 
 $$

Similarly, from Zagier's trick, we have 
\begin{align*}
     F( \varkappa(d,2),1)&= \frac{1}{\pi i} \int_{\frac{1+i}2}^{\frac{-1+i}2} f_{8.4.a.a}(\tau/2)d\tau\\
     &= \frac{1}{\pi i} \int_{\frac{1+i}2}^{\frac 12+i\infty} f_{8.4.a.a}(\tau/2)d\tau-\frac{1}{\pi i}\int_{\frac{-1+i}2}^{-\frac 12+i\infty}f_{8.4.a.a}(\tau/2)d\tau. 
\end{align*}
Recall that $z_2(\tfrac{1+i}2)=z_2(\tfrac{-1+i}2)=1/2$ and $f_{8.4.a.a}(\tau\pm 1/4)=\pm i f_{16.4.a.a}(\tau)$. This identity is basically the same as the one obtained from the Euler integral. 
\end{example}

 In the preprint \cite{HTY}, the authors use the relation between the $L$-value and Eichler integral of the form $f_{12.4.a.a}$ to express $L(f_{12.4.a.a},2)$ in terms of hypergeometric evaluations with $(d,r)=(3,1/2)$ in \eqref{eqn: Euler}. 
 \begin{example}[\cite{HTY}]
 \begin{align*}
     \frac{24\sqrt3}\pi &L(f_{12.4.a.a},2)\\
     &=\frac{\Gamma(\frac 16)\Gamma(\frac 13)^2}{\Gamma(\frac 56)\Gamma(\frac 23)^3}
    \pFq43{\frac 13&\frac 13&\frac 13& \frac 13}{&\frac 23&\frac 56&\frac56}1
    -108\frac{\Gamma(\frac 56)\Gamma(\frac 23)^2}{\Gamma(\frac 16)\Gamma(\frac 13)^3}
       \pFq43{\frac 23&\frac 23&\frac 23& \frac 23}{&\frac 43&\frac 76&\frac76}1.
 \end{align*}
   The ${}_4F_3(1)$'s in the equality are arising from the solutions around the singularity $t=\infty$ of the differential equation satisfied by $F(\varkappa_3(3),t)$. 
To get the identity, the main idea is to consider the integral 
$$
  \int_0^\infty \left(\frac t{1-t}\right)^{1/2}F(\varkappa_3(3),t) \frac{dt} t =2\pi i \int_{\frac{3+\sqrt{-3}}6}^\infty g(\tau/2)d\tau,
$$
which in turn is a multiple of $L(f_{12.4.a.a},2)$. 
\end{example}

In a similar manner, the $L$-value of $f_{9.4.a.a}(\tau)=\eta(3\tau)^8$ can be obtained using Euler integral (\cite[Family (2)]{Barman_MaityEHMM}) and the  Barnes integral representation of $F(\varkappa_3(2),t)$: 
\begin{align*}
     L(\eta(3\tau)^8,1)
     &=\frac {\sqrt \pi}6\left(\frac{\Gamma(\frac 13)}{\Gamma(\frac 56)} \pFq43{\frac 12&\frac 12&\frac 12& \frac 13}{&1&1&\frac56}1
    -\frac{\Gamma(\frac 23)^3}{\pi\Gamma(\frac 76)^3}
       \pFq43{\frac 23&\frac 23&\frac 23& \frac 12}{&\frac 76&\frac 76&\frac76}1
    \right)\\
    &=\frac {\sqrt{\pi}}{12}\left(\frac{\Gamma(\frac 13)^3}{\pi\Gamma(\frac 56)^3}
       \pFq43{\frac 13&\frac 13&\frac 13& \frac 12}{&\frac 56&\frac 56&\frac56}1-\frac{\Gamma(\frac 23)}{\Gamma(\frac 76)} \pFq43{\frac 12&\frac 12&\frac 12& \frac 23}{&1&1&\frac76}1
    \right). 
\end{align*}

{ Shimura’s work \cite{Shimura-special-zeta} established, prior to Deligne’s formulation of his general conjecture \cite{deligne-l-values}, the expected algebraicity and period relations for critical $L$-values of modular forms. In this setting, the critical value $L(f,1)$ is expressed, up to an algebraic factor, in terms of one of two fixed periods $c^{\pm}$. See \cite{deligneraghuram} for a modern exposition of these ideas. Deligne’s conjecture applies to a far broader class of geometric motives than those arising from modular forms, however, and is expected to apply to the $L$-functions of hypergeometric motives as well. The periods of hypergeometric motives are known to involve classical hypergeometric functions. Accordingly, hypergeometric evaluations of $L$-functions provide a natural testing ground for Deligne’s conjecture. In examples such as the one above, when the modular form has complex multiplication, the transcendental part of the period is given by a fixed quotient of gamma values; see \cite{Gross}. For the example above, this relation is made explicit in Appendix \ref{lvalueappendix}. In \cite{rosenmixedI}, it is explained in detail how classical hypergeometric evaluation formulas are related to $L$-values of CM modular forms. In particular, in certain special cases, a single hypergeometric series is related to multiple modular forms with distinct CM discriminants. The examples above suggest that further relations between CM $L$-values and hypergeometric evaluations remain to be discovered.}

\medskip
Here are some additional future directions.
\begin{itemize}
    \item How to use the partition \eqref{eq:partition} to obtain relations between periods of modular forms and hypergeometric values as described above?
    \item Can we use the EHMM to prove other remaining conjectures in \cite{gugiatti2024hypergeometriclocalsystemsmathbbq,Long18}?

\end{itemize}

\bigskip

\def\Gfam{\mathfrak G}
\appendix
\renewcommand{\thesection}{A.\arabic{section}}
\section{Explicit Modular Forms}\label{EHMMappendix}

In this section, we provide more details on the explicit modular forms computed using the EHMM for \Cref{lemma: constructingforms} in the cases of $d=2,4$. The weight four cusp forms $\BG_{d}(i/e)(e \tau)$ constructed in \Cref{lemma: constructingforms} are a key ingredient in the construction of the $g_{d,e}$ eigenforms. We now extend the explicit discussion of the $\BG_{3}$ functions from \Cref{constructingMF} to the $d = 2,4$ cases.

In the $d = 2$ case, we compute the cusp forms
$$\BG_{2}(i/e)(e \tau) := \eta(e\tau)^{8i/e}\eta(2e\tau)^{8(2-3i/e)}\eta(4e\tau)^{8(2i/e-1)}$$

Now, to ensure congruence cusp forms, we restrict to the rational numbers $i/e \in (0,1)$ such that $8i/e$ is an integer. These seven rational numbers are grouped into three families. Similarly, in the $d=4$ case we find the cusp forms 
$$\BG_{4}(i/e)(e \tau) = E_{2,2}(e\tau)^{2} \cdot \frac{\eta(e\tau)^{24(1-i/e)}\eta(2e\tau)^{24i/e}}{\eta(e \tau)^{24} + 64 \eta(2 e \tau)^{24}}$$

Now, we restrict to the rational numbers $i/e \in (0,1)$ such that $24i/e$ is an integer. These twenty-three rational numbers are grouped into seven families. All $\Gfam_{d,e}$ families for $d=2,3,4$ are given in the tables below, with an explicit eigenform for each family. Note that there are $\varphi(e)$ normalized Hecke eigenforms, as described in \Cref{lemma: constructingforms}; we will list one of them.

\renewcommand{\arraystretch}{1.18}

\begin{table}[ht]
\caption{The $\Gfam_{2,e}$ families with a chosen element}\label{G2table}
\vspace{2mm}
\begin{center}
    \begin{tabular}{|c|c|c|c|c|}
    \hline
    \textrm{Family} & \textrm{$r$ values} & \textrm{Eigenform} & \textrm{Level} & \textrm{Character } $\xi_{d,e}$ \\ \hline
$\Gfam_{2,2}$ & $1/2$ &  $f_{8.4.a.a} = \BG_{2}(1/2)(2 \tau)$ & $8$ & $\left(\frac{1}{\cdot}\right)$ \\ \hline
$\Gfam_{2,4}$ &$1/4,3/4$ & $f_{16.4.a.a} = \BG_{2}(1/4)(4 \tau)$ & $16$ & $\left(\frac{1}{\cdot}\right)$\\
&&$+ 4 \BG_{2}(3/4)(4 \tau)$ && \\ \hline
$\Gfam_{2,8}$ & $1/8,3/8$ & $f_{32.4.b.a} = \BG_{2}(1/8)(8 \tau)$ & $32$ & $\left(\frac{2}{\cdot}\right)$ \\ 
& $5/8,7/8$ &$-2 \sqrt{-7}\BG_{2}(3/8)(8 \tau)$ &&\\
&&$-4 \sqrt{-7} \BG_{2}(5/8)(8 \tau)$&&\\
&&$+8 \BG_{2}(7/8)(8 \tau)$&&\\
\hline
    \end{tabular}
\end{center}
\end{table} 
\renewcommand{\arraystretch}{1}

\renewcommand{\arraystretch}{1.02}

\begin{table}[ht]
\caption{The $\Gfam_{3,e}$ families with a chosen element}\label{G3table}
\vspace{2mm}
\begin{center}
    \begin{tabular}{|c|c|c|c|c|}
    \hline
    \textrm{Family} & \textrm{$r$ values} & \textrm{Eigenform} & \textrm{Level} & \textrm{Character } $\xi_{d,e}$ \\ \hline
$\Gfam_{3,2}$ & $1/2$ & $f_{12.4.a.a} = \BG_{3}(1/2)(2 \tau)$ & $12$ & $\left(\frac{1}{\cdot}\right)$ \\ \hline
$\Gfam_{3,3}$ &$1/3,2/3$ & $f_{27.4.a.b} = \BG_{3}(1/3)(3 \tau)$ & $27$ & $\left(\frac{1}{\cdot}\right)$\\
&&$+ 3 \BG_{3}(2/3)(3 \tau)$ && \\ \hline
$\Gfam_{3,4}$ & $1/4,3/4$ & $f_{48.4.c.a} = \BG_{3}(1/4)(4 \tau)$ & $48$ & $\left(\frac{3}{\cdot}\right)$\\
&&$-3 \sqrt{-3} \BG_{3}(3/4)(4 \tau)$ && \\ \hline
$\Gfam_{3,6}$ & $1/6,5/6$ & $f_{108.4.a.d} = \BG_{3}(1/6)(6 \tau)$ & $108$ & $\left(\frac{1}{\cdot}\right)$\\
&&$+ 9 \BG_{3}(5/6)(6 \tau)$ &&\\ \hline
$\Gfam_{3,12}$ & $1/12,5/12$ & $f_{432.4.c.f} = \BG_{3}(1/12)(12 \tau)$ & $432$ & $\left(\frac{3}{\cdot}\right)$ \\ 
& $7/12, 11/12$ &$+3\sqrt{-17} \cdot \BG_{3}(5/12)(12 \tau)$ &&\\
&&$+ 3\sqrt{-51} \cdot \BG_{3}(7/12)(12 \tau)$&&\\
&&$+ 9\sqrt{3} \BG_{3}(11/12)(12 \tau)$&&\\
\hline

    \end{tabular}
\end{center}
\end{table}

\renewcommand{\arraystretch}{1}

\renewcommand{\arraystretch}{1.18}

\begin{table}[ht]

\caption{The $\Gfam_{4,e}$ families with a chosen element}\label{G4table}
\vspace{2mm}
\begin{center}
    \begin{tabular}{|c|c|c|c|c|}
    \hline
    \textrm{Family} & \textrm{$r$ values} & \textrm{Eigenform} & \textrm{Level} & \textrm{Character } $\xi_{d,e}$ \\ \hline
$\Gfam_{4,2}$ & $1/2$ & $f_{8.4.a.a} = \BG_{4}(1/2)(2 \tau)$ & $8$ & $\left(\frac{1}{\cdot}\right)$ \\ \hline
$\Gfam_{4,3}$ &$1/3,2/3$ & $f_{9.4.a.a}(\tau) = \BG_{4}(1/3)(3 \tau)$ & $18$ & $\left(\frac{1}{\cdot}\right)$\\
&&$f_{9.4.a.a}(2 \tau) = \BG_{4}(2/3)(3 \tau)$ && \\ \hline
$\Gfam_{4,4}$ & $1/4,3/4$ & $f_{32.4.a.c} = \BG_{4}(1/4)(4 \tau)$ & $32$ & $\left(\frac{1}{\cdot}\right)$\\
&&$+ 8 \BG_{4}(3/4)(4 \tau)$ && \\ \hline
$\Gfam_{4,6}$ & $1/6,5/6$ & $f_{72.4.a.d} = \BG_{4}(1/6)(6 \tau)$ & $72$ & $\left(\frac{1}{\cdot}\right)$\\
&&$+ 16 \BG_{4}(5/6)(6 \tau)$ &&\\ \hline
$\Gfam_{4,8}$ & $1/8,3/8$ & $f_{128.4.b.e} = \BG_{4}(1/8)(8 \tau)$ & $128$ & $\left(\frac{2}{\cdot}\right)$ \\ 
& $5/8,7/8$ &$+2\sqrt{-10} \cdot \BG_{4}(3/8)(8 \tau)$ &&\\
&&$+ 8\sqrt{-5} \cdot \BG_{4}(5/8)(8 \tau)$&&\\
&&$+ 16\sqrt{2} \cdot \BG_{4}(7/8)(8 \tau)$&&\\
\hline
$\Gfam_{4,12}$ & $1/12,5/12$ & $f_{288.4.a.\ell} = \BG_{4}(1/12)(12 \tau)$ & $288$ & $\left(\frac{1}{\cdot}\right)$\\
& $7/12,11/12$ &$-4 \sqrt{13} \BG_{4}(5/12)(12 \tau)$ &&\\
&&$+8 \sqrt{13} \BG_{4}(7/12)(12 \tau)$&&\\
&&$-32 \BG_{4}(11/12)(12 \tau)$&&\\
\hline
$\Gfam_{4,24}$ & $1/24, 5/24$ & $f_{1152.4.d.q} = \BG_{4}(1/24)(24 \tau)$ & $1152$ & $\left(\frac{2}{\cdot}\right)$\\
& $7/24,11/24$ &$+ 2\sqrt{-35} \cdot \BG_{4}(5/24)(24 \tau)$&&\\
& $13/24, 17/24$ &$ -2 \sqrt{110} \cdot \BG_{4}(7/24)(24 \tau)$&&\\
& $19/24, 23/24$ &$+ 4 \sqrt{-154} \cdot \BG_{4}(11/24)(24 \tau)$&&\\
&&$+8\sqrt{-77} \cdot \BG_{4}(13/24)(24 \tau)$&&\\
&&$-16\sqrt{55} \cdot \BG_{4}(17/24)(24 \tau)$&&\\
&&$-16 \sqrt{-70} \cdot \BG_{4}(19/24)(24 \tau)$&&\\
&&$ -280 \sqrt{2} \cdot \BG_{4}(23/24)(24 \tau)$&&\\
\hline

    \end{tabular}
\end{center}
\end{table}

\renewcommand{\arraystretch}{1}

\FloatBarrier

\section{\texorpdfstring{$L$}{L}-values of the CM cusp forms}\label{lvalueappendix}
For the CM forms associated with  $F( \varkappa(3,4),1)$, we have the following $L$-values. 
\begin{theorem}  Let $\Omega=\frac{\G(1/3)^6}{\G(2/3)^3}$. 
 
We have
$$
    L(\eta(3\tau)^8,3)= \frac{\Omega}{162}=\frac{2\pi^2}9 L(\eta(3\tau)^8,1)= \frac{2\sqrt 3\pi}9 L(\eta(3\tau)^8,2).
$$
Let $g_{\pm} = \BG_{3}(1/4)(4 \tau)\pm3 \sqrt{-3} \BG_{3}(3/4)(4 \tau)$. Then
$$ 
L(g_{-},3) =\frac{ (1-i)\zeta_{12}}{2^{13/2}\cdot 3^{3/4}} \Omega=\frac{\pi^2}{24 i}L(g_{+},1)$$

and

$$L(g_{-},2)= \frac{(1-i)}{2^{9/2}\cdot 3^{3/4}}\frac{\Omega}{\pi}.
$$
\end{theorem}

\begin{proof}
For the CM form $\eta(3\tau)^8$, according to \cite{LLT} and the properties of  period polynomials in \cite[Section 11]{Cohen-Stromberg} or \cite{HTY}, one can obtain
$$
\frac{\Omega}{162}=L(\eta(3\tau)^8,3)= \frac{2\pi^2}9 L(\eta(3\tau)^8,1)= \frac{2\sqrt 3\pi}9 L(\eta(3\tau)^8,2).
$$
For the CM forms in the Hecke orbit  $48.4.c.a$, their $q$-expansions are
\begin{align*}
g_{\pm} =& 2\sum_{m, n \in \Z} (4m+1+4n\sqrt{-3}) q^{(4m+1)^2+48n^2}- \sum_{m, n \in \Z} (4m+1+2n\sqrt{-3}) q^{(4m+1)^2+12n^2}\\
         &\pm \sum_{m, n \in \Z} (2m+(4n+1)\sqrt{-3}) q^{4m^2+3(4n+1)^2}\\
         &\mp 2\sum_{m, n \in \Z} (4m+(4n+1)\sqrt{-3}) q^{16m^2+3(4n+1)^2}
\end{align*}
Here  $g_{-} = \BG_{3}(1/4)(4 \tau)-3 \sqrt{-3} \BG_{3}(3/4)(4 \tau)$ in our notation. 
From the $q$-expansion of $g_{-}$, we can derive that 
$$
L(g,3) := 4\pi^3i \left(\left(\frac1{32\cdot 9}\mathcal G(s/12)-\frac1{8\cdot 9}\mathcal G(s/6)\right)\sqrt{-3}-\frac18\mathcal G(s/2)+\frac1{32}\mathcal G(s/4)\right),
$$
where $s=\sqrt{-3}$ and 
$$
\mathcal G(\tau)=\sum_{m, n \in \Z}\frac{1}{[(4m+1)\tau+n]^3} = - \frac1{16}\theta_3(2\tau)^2\theta_2(2\tau)^4=-\frac{\eta(2\tau)^6\eta(4\tau)^4}{\eta(\tau)^4}. 
$$
Therefore, the special value follows the known special values and transformation laws of $\eta(\tau)$, and that 
$$
\eta(s)^6=\frac{1}{2^53^{4/3}}\frac{\Omega}{\pi^3},
$$
(see \cite{LLT}). We will omit the details and list the exact $\eta$-values. 
$$
\begin{array}{c|c}
\tau_0&\eta(\tau_0)     /\eta(s)\\ \hline
s/12&2^{1/24}3^{1/4}(495\sqrt 6 - 495\sqrt 2+360\sqrt 3-1136)^{1/24}\\
s/6&2^{1/24}3^{1/4}(3\sqrt3-5)^{1/12}\\
s/4&(1980\sqrt 2 + 1980\sqrt 6- 1440\sqrt 3-4544 )^{1/24}\\
s/3&3^{1/4}\\
2s/3&\frac{\sqrt 2}{3^{1/4}}(104+60\sqrt3)^{1/24}\\
s/2&(104+60\sqrt3)^{1/24}\\
2s&\frac{(3\sqrt3-5)^{1/12}}{2^{11/24}}
\end{array}
$$

To find the relation between $L(g_{-},2)$ and $L(g_{-},3)$, we note that the normalizer of $\G_0(48)$ in $\SL_2(\R)$ is generated by 
$$
\left\langle\M 0{1/4}01, \frac 1{\sqrt{48}}\M0{-1}{48}0\right\rangle\simeq (2,6,\infty). 
$$
Applying similar properties  of  period polynomials described in \cite[Section 11]{Cohen-Stromberg} or \cite{HTY}, we can see that 
$$
12L(g_{-},3)-3\pi iL(g_{-},2)-\frac{\pi^2}2L(g_{-},1) =0.
$$
From the functional equation that $\pi^2 L(g_{-},1)=-24iL(g_{+},3)=-24i\ol{L(g_{-},3)}$, one can deduce the value $L(g_{-},2)$. 
\end{proof}

It will be interesting to know if the $L$-values are connected with the relevant hypergeometric values via Euler integral,  Barnes integral representations, or Zagier's trick.  For instance, using Zagier's trick, one can show that 
$$
 F( \varkappa(4,3),1)=12\sqrt 3 \cdot \left(\int_{1/\sqrt{18}}^\infty\eta(3it)^8dt +4 \int_{1/\sqrt{18}}^\infty\eta(6it)^8dt\right),
$$
where the two forms are the cusp forms on level $18$ as in Table \ref{G4table}. Does the information arising from the Eichler integral of $\eta(3\tau)^8$ yield any relation to its $L$-values? We leave this question to interested readers.

\input{Notation}

\bibliographystyle{plain}
\bibliography{ref}

\end{document}

%% file: Notation.tex
\section{Notation}\label{sec:notation}

\begin{itemize}
    \item CFGL: An abbreviation for the commutative formal group law.
    \item REAB: An abbreviation for Ramanujan's theory of elliptic functions to alternative bases.
    \item $M(d,e)$ : The least common multiple of $d$ and $e$. 
    \item $M_{d}$ : Positive integers which depend on $d$. Given explicitly in \Cref{thm: Galois version}.
    \item $\varphi(\cdot)$ : The Euler totient function.
    \item $(a)_{k}:= a(a+1) \cdots (a+k-1)$ : The Pochhammer symbol of $a$ with index $k$.
     \item For $\varkappa=\{\{a_1,\dots ,a_n\},\{b_1,\dots ,b_n\}\}$ define $A_{\varkappa}(k)=\prod_{i=1}^n\frac{(a_i)_k}{(b_i)_k}$.
    \item $F(\varkappa,z)=\sum_{k\ge0} A_{\varkappa}(k)z^k$ : classical hypergeometric series.
    \item $F(\varkappa,z)_{m}= \sum_{k=0}^m A_{\varkappa}(k)z^k$ : truncated hypergeometric function. 
    \item $\G(\cdot)$ is the gamma function, $\G_p(\cdot)$ is the $p$-adic Gamma function.
    \item For $\varkappa_1=\{\alpha_1,\beta_1\}$, $\varkappa_2=\{\alpha_2,\beta_2\}$, define $\varkappa_1\star\varkappa_2:=\{\alpha_1\cup \alpha_2,\beta_1\cup \beta_2\}$.
    \item $\varkappa_3(d):=\{\{1/d,1-1/d,1/2\},\{1,1,1\}\}$
    \item $\varkappa_2(d):=\{\{1/d,1-1/d\},\{1,1\}\}$
   
    \item $\varkappa_{alg}(a):=\{\{a,1-a\},\{1,1/2\}\}$
    \item $\varkappa(d;r,s) := \{\{1/d,(d-1)/d,r\},\{1,1,s\}\}$
     \item $\varkappa(d,e):=\varkappa_3(d)\star\varkappa_{alg}(1/e)$ 
    \item $\varkappa_{cont}(a):=\{\{a,a+1/2\},\{2a,1\}\}$; when $a=1/2$, it degenerates to $\{\{\frac12\},\{1\}\}$
   
    \item $t_d$ : a Hauptmodul of $(2,\tfrac{2d}{d-2},\infty)$, see Table \ref{tab:bigger triangle}
    \item $C_1(d)$ : The first $q$ coefficient of $t_d$.
    \item $z_d$ : A fixed modular function on $\Gamma_1(6-d)$  so that $t_d=4z_d(1-z_d)$, $\sqrt{1-t_d}=1-2z_d$, $\frac{1-\sqrt{1-t_d}}{2}=z_d$.
   \item ${\BG}_d(i/e)$: a fixed weight four modular form, see \eqref{eq:g(d,e,i)}
   
   \item $\Gfam_{d,e}$ : a vector space of dimension $\varphi(e)$ containing the $\BG_{d}(i/e)(e \tau)$ functions.
    \item $g_{d,e}$ : Normalized Hecke eigenforms of weight four given in  Tables \ref{G2table}, \ref{G3table}, \ref{G4table}.
    \item $f_{d,e}^{\sharp}$ : Normalized Hecke eigenforms of weight four in \Cref{thm: Galois version}.
    \item $\beta_{d,e,i}$ : Constants in \Cref{lemma: constructingforms} depending on $d,e,$ and $i$.
    \item $R:= \Z_p[C_1(d)^{1/e}, \beta_{d,e,j}: j\in  (\Z/e\Z)^\times]$ : The $p$-adic ring in \eqref{eq:R}.
    \item $\sigma$ : An automorphism of $R$ such that for each $a \in R$, $\sigma(a) \equiv a^{p} \pmod{pR}$.
    \item $[a]_{0}$ : The first $p$-adic digit for a given $a \in \Z_{p}$.
    \item $a':= (a + [-a]_{0})/p$ : The Dwork dash operation for a given $a \in \Z_{p}$.
    \item $\zeta_{M}$ : A fixed primitive $M$-th root of unity.
    \item $\mathfrak{p}$ : A prime ideal of $\Z[\zeta_{M}]$.
    \item $\kappa_{\mathfrak{p}}:= \Z[\zeta_{M}]/\mathfrak{p}$ : A residue field of $\Z[\zeta_{M}]$ at $\mathfrak{p}$.
    \item $\textrm{St}_{\varkappa} := \{c\in (\Z/M\Z)^\times \mid c\cdot \varkappa=\varkappa\}$ : The stablizer of the datum $\varkappa$, where $M$ is the least common denominator of the entries in $\varkappa$.
    \item $K_{\varkappa}$ : The fixed field of $\Q(\zeta_{M})$ under $\textrm{St}_{\varkappa}$. 
    \item $\BH(\varkappa;t;{ \wp})$ : The hypergeometric character sums defined in \eqref{eq:Hpdef}.
    \item $G_{K}:= \Gal(\overline{K}/K)$ : The absolute Galois group of a field $K$.
\item $\eta_{\varkappa,\ell,1}$ : The hypergeometric representation of $G_{\Q(\zeta_{M})}$ associated with the datum $\varkappa$ and the parameter $1$. Discussed in \cite{LLL26}.
\item $\eta_{\varkappa,\ell,1}^{\mathrm{BCM}}$ : The
semisimplification of the hypergeometric representation of $G_{\Q}$ associated with the datum
$\varkappa$ and the parameter $1$, as defined in \cite{LLL26}. 

\item $\rho_{f}$ : An $\ell$-adic representation of $G_{\Q}$ attached to the Hecke eigenform $f$.
\item $\rho_{d,e,\ell}$ : A two-dimensional representation in \eqref{eq:eta-decomp}.

\item $\mu_{\varkappa,t,p}$ : The $p$-adic unit root attached to the datum $\varkappa$ and parameter $t$.

\item $\varepsilon_{d,e}(p)$ : A sign which depends on $d,e,$ and $p$ that is determined by a congruence in \Cref{thm:H-version}.

\item $\varphi_{d,e}$: A finite order character in \eqref{eq:eta-decomp}. Described in the proof.

    \item $\epsilon_{\ell}$ : The $\ell$-adic cyclotomic character.
    \item $\iota_{\mathfrak{p}}(\cdot /M)$: an $M$-th residue symbol with modulus $\mathfrak{p}$. Defined in \eqref{eq:iotadef}.
    \item $\chi_{d}$ :  The Dirichlet character corresponding to $\mathbb Q(\sqrt d)$. 
    \item $\varsigma(e)$ : Quadratic Dirichlet characters depending on $e$, see \Cref{thm: Galois version}.
    \item $\xi_{d,e}$ : Dirichlet characters depending on $d,e$, see \Cref{lemma: constructingforms}.
   
\end{itemize}

%% file: ref.bib
@unpublished{HTY,
  author  = {Hsu, Che-Wei and Tu, Fang-Ting and Yang, Yifan},
  title   = {Hypergeometric evaluations of {$L$}-values of weakly holomorphic modular forms and harmonic Maass forms},
  note    = {In preparation},
  year    = {},
}

@article {RWZ,
    AUTHOR = {Rogers, M. and Wan, J. G. and Zucker, I. J.},
     TITLE = {Moments of elliptic integrals and critical {$L$}-values},
   JOURNAL = {Ramanujan J.},
  FJOURNAL = {Ramanujan Journal. An International Journal Devoted to the
              Areas of Mathematics Influenced by Ramanujan},
    VOLUME = {37},
      YEAR = {2015},
    NUMBER = {1},
     PAGES = {113--130},
      ISSN = {1382-4090,1572-9303},
   MRCLASS = {11F03 (11M41 33C20 33C75 33E05)},
  MRNUMBER = {3338042},
MRREVIEWER = {Qiao\ Zhang},
       DOI = {10.1007/s11139-014-9584-5},
       URL = {https://doi.org/10.1007/s11139-014-9584-5},
}

@incollection {kedlaya,
    AUTHOR = {Kedlaya, Kiran S.},
     TITLE = {Frobenius structures on hypergeometric equations},
 BOOKTITLE = {Arithmetic, geometry, cryptography, and coding theory 2021},
    SERIES = {Contemp. Math.},
    VOLUME = {779},
     PAGES = {133--158},
 PUBLISHER = {Amer. Math. Soc., [Providence], RI},
      YEAR = {[2022] \copyright 2022},
      ISBN = {978-1-4704-6794-4},
   MRCLASS = {33C80 (12H25 33C20)},
  MRNUMBER = {4445774},
MRREVIEWER = {Fana\ Tangara},
       DOI = {10.1090/conm/779/15673},
       URL = {https://doi.org/10.1090/conm/779/15673},
}

@article {Fedorov18,
    AUTHOR = {Fedorov, Roman},
     TITLE = {Variations of {H}odge structures for hypergeometric
              differential operators and parabolic {H}iggs bundles},
   JOURNAL = {Int. Math. Res. Not. IMRN},
  FJOURNAL = {International Mathematics Research Notices. IMRN},
      YEAR = {2018},
     PAGES = {5583--5608},
      ISSN = {1073-7928,1687-0247},
   MRCLASS = {32G20 (14D07 14H60 33C90)},
  MRNUMBER = {3862114},
MRREVIEWER = {Javier\ A.\ Fern\'{a}ndez},
       DOI = {10.1093/imrn/rnx044},
       URL = {https://doi.org/10.1093/imrn/rnx044},
}

@article {HypMot,
    AUTHOR = {Roberts, David P. and Rodriguez Villegas, Fernando},
     TITLE = {Hypergeometric Motives},
   JOURNAL = {Notices Amer. Math. Soc.},
  FJOURNAL = {Notices of the American Mathematical Society},
    VOLUME = {69},
      YEAR = {2022},
    NUMBER = {6},
     PAGES = {914--929},
      ISSN = {0002-9920},
   MRCLASS = {01A70 (55-03 57-03)},
  MRNUMBER = {1691561},
}

@book {KatzRigid,
    AUTHOR = {Katz, Nicholas M.},
     TITLE = {Rigid local systems},
    SERIES = {Annals of Mathematics Studies},
    VOLUME = {139},
 PUBLISHER = {Princeton University Press, Princeton, NJ},
      YEAR = {1996},
     PAGES = {viii+223},
      ISBN = {0-691-01118-4},
   MRCLASS = {14F20 (14F05)},
  MRNUMBER = {1366651},
MRREVIEWER = {Abdellah Mokrane},
       DOI = {10.1515/9781400882595},
       URL = {https://doi-org.libezp.lib.lsu.edu/10.1515/9781400882595},
}

@book {KatzESDE,
    AUTHOR = {Katz, Nicholas M.},
     TITLE = {Exponential sums and differential equations},
    SERIES = {Annals of Mathematics Studies},
    VOLUME = {124},
 PUBLISHER = {Princeton University Press, Princeton, NJ},
      YEAR = {1990},
     PAGES = {xii+430},
      ISBN = {0-691-08598-6; 0-691-08599-4},
   MRCLASS = {14D10 (11L03 11T23 14G15)},
  MRNUMBER = {1081536},
MRREVIEWER = {Hernando Enrique Sierra-Morales},
       DOI = {10.1515/9781400882434},
       URL = {https://doi-org.libezp.lib.lsu.edu/10.1515/9781400882434},
}

@misc{LLL26,
     
      author={Wen-Ching Winnie Li and Tong Liu and Ling Long},
       title={The arithmetic of hypergeometric {G}alois representations}, 
      year={preprint 2026},
      eprint={},
      archivePrefix={},
      primaryClass={},
      url={}, 
}

@misc{grove2025hypergeometricmodularityconjecturesdawsey,
 author={Brian Grove},
      title={On {S}ome {H}ypergeometric {M}odularity {C}onjectures of {D}awsey and {M}c{C}arthy}, 
      year={2025, arXiv 2507.19971},
}

@article {lennon1,
    AUTHOR = {Lennon, Catherine},
     TITLE = {Trace formulas for {H}ecke operators, {G}aussian
              hypergeometric functions, and the modularity of a threefold},
   JOURNAL = {J. Number Theory},
  FJOURNAL = {Journal of Number Theory},
    VOLUME = {131},
      YEAR = {2011},
    NUMBER = {12},
     PAGES = {2320--2351},
      ISSN = {0022-314X},
   MRCLASS = {11F25 (11G20 11G40 11T24 33E50)},
  MRNUMBER = {2832827},
MRREVIEWER = {Jenny G. Fuselier},
       DOI = {10.1016/j.jnt.2011.05.005},
       URL = {https://doi-org.libezp.lib.lsu.edu/10.1016/j.jnt.2011.05.005},
}

@article {BBG-Ramanujan,
    AUTHOR = {Berndt, Bruce C. and Bhargava, S. and Garvan, Frank G.},
     TITLE = {Ramanujan's theories of elliptic functions to alternative
              bases},
   JOURNAL = {Trans. Amer. Math. Soc.},
  FJOURNAL = {Transactions of the American Mathematical Society},
    VOLUME = {347},
      YEAR = {1995},
    NUMBER = {11},
     PAGES = {4163--4244},
      ISSN = {0002-9947},
   MRCLASS = {33E05 (11F27 33C05 33D10)},
  MRNUMBER = {1311903},
MRREVIEWER = {J. Borwein},
       DOI = {10.2307/2155035},
       URL = {https://doi-org.libezp.lib.lsu.edu/10.2307/2155035},
}

@article {StillerHF,
    AUTHOR = {Stiller, P. F.},
     TITLE = {Classical automorphic forms and hypergeometric functions},
   JOURNAL = {J. Number Theory},
  FJOURNAL = {Journal of Number Theory},
    VOLUME = {28},
      YEAR = {1988},
    NUMBER = {2},
     PAGES = {219--232},
      ISSN = {0022-314X},
   MRCLASS = {11F12 (11F20 33A30)},
  MRNUMBER = {927661},
MRREVIEWER = {Jannis A. Antoniadis},
       DOI = {10.1016/0022-314X(88)90067-4},
       URL = {https://doi-org.libezp.lib.lsu.edu/10.1016/0022-314X(88)90067-4},
}

@book {AAR,
    AUTHOR = {Andrews, George E. and Askey, Richard and Roy, Ranjan},
     TITLE = {Special functions},
    SERIES = {Encyclopedia of Mathematics and its Applications},
    VOLUME = {71},
 PUBLISHER = {Cambridge University Press, Cambridge},
      YEAR = {1999},
     PAGES = {xvi+664},
      ISBN = {0-521-62321-9; 0-521-78988-5},
   MRCLASS = {33-01 (33-02)},
  MRNUMBER = {1688958},
MRREVIEWER = {Bruce C. Berndt},
       DOI = {10.1017/CBO9781107325937},
       URL = {https://doi.org/10.1017/CBO9781107325937},
}

@article {Ahlgren-Ono-CalabiYau,
    AUTHOR = {Ahlgren, Scott and Ono, Ken},
     TITLE = {Modularity of a certain {C}alabi-{Y}au threefold},
   JOURNAL = {Monatsh. Math.},
  FJOURNAL = {Monatshefte f\"ur Mathematik},
    VOLUME = {129},
      YEAR = {2000},
    NUMBER = {3},
     PAGES = {177--190},
      ISSN = {0026-9255},
     CODEN = {MNMTA2},
   MRCLASS = {11G40 (11F20 11F23 11F80 14J32)},
  MRNUMBER = {1746757 (2001b:11059)},
MRREVIEWER = {Noriko Yui},
       DOI = {10.1007/s006050050069},
       URL = {http://dx.doi.org/10.1007/s006050050069},
}

@article {BCM,
    AUTHOR = {Beukers, Frits and Cohen, Henri and Mellit, Anton},
     TITLE = {Finite hypergeometric functions},
   JOURNAL = {Pure Appl. Math. Q.},
  FJOURNAL = {Pure and Applied Mathematics Quarterly},
    VOLUME = {11},
      YEAR = {2015},
    NUMBER = {4},
     PAGES = {559--589},
      ISSN = {1558-8599},
   MRCLASS = {11T24 (11L05 14G05 14M25 33C20 33C80)},
  MRNUMBER = {3613122},
MRREVIEWER = {Daniel Barsky},
       DOI = {10.4310/PAMQ.2015.v11.n4.a2},
       URL = {http://dx.doi.org/10.4310/PAMQ.2015.v11.n4.a2},
}

@article {Stienstra-Beukers,
    AUTHOR = {Stienstra, Jan and Beukers, Frits},
     TITLE = {On the {P}icard-{F}uchs equation and the formal {B}rauer group
              of certain elliptic {$K3$}-surfaces},
   JOURNAL = {Math. Ann.},
  FJOURNAL = {Mathematische Annalen},
    VOLUME = {271},
      YEAR = {1985},
    NUMBER = {2},
     PAGES = {269--304},
      ISSN = {0025-5831},
   MRCLASS = {14L05 (11G35 14J20 14J28)},
  MRNUMBER = {783555},
MRREVIEWER = {Thomas Zink},
       DOI = {10.1007/BF01455990},
       URL = {https://doi-org.libezp.lib.lsu.edu/10.1007/BF01455990},
}

@book {BB,
    AUTHOR = {Borwein, Jonathan M. and Borwein, Peter B.},
     TITLE = {Pi and the {AGM}},
    SERIES = {Canadian Mathematical Society Series of Monographs and
              Advanced Texts},
    VOLUME = {4},
      NOTE = {A study in analytic number theory and computational complexity,
              Reprint of the 1987 original,
              A Wiley-Interscience Publication},
 PUBLISHER = {John Wiley \& Sons, Inc., New York},
      YEAR = {1998},
     PAGES = {xvi+414},
      ISBN = {0-471-31515-X},
   MRCLASS = {11Y60 (11B65 68Q25)},
  MRNUMBER = {1641658},
}

@article {BBG,
    AUTHOR = {Borwein, Jonathan M. and Borwein, Peter B. and Garvan, Frank. G.},
     TITLE = {Some cubic modular identities of {R}amanujan},
   JOURNAL = {Trans. Amer. Math. Soc.},
  FJOURNAL = {Transactions of the American Mathematical Society},
    VOLUME = {343},
      YEAR = {1994},
    NUMBER = {1},
     PAGES = {35--47},
      ISSN = {0002-9947},
   MRCLASS = {11B65 (11F27 33D10)},
  MRNUMBER = {1243610},
MRREVIEWER = {George E. Andrews},
       DOI = {10.2307/2154520},
       URL = {https://doi-org.libezp.lib.lsu.edu/10.2307/2154520},
}

@article {Dwork,
    AUTHOR = {Dwork, Bernard},
     TITLE = {{$p$}-adic cycles},
   JOURNAL = {Inst. Hautes \'Etudes Sci. Publ. Math.},
  FJOURNAL = {Institut des Hautes \'Etudes Scientifiques. Publications
              Math\'ematiques},
      YEAR = {1969},
     PAGES = {27--115},
      ISSN = {0073-8301},
   MRCLASS = {14G20},
  MRNUMBER = {0294346},
MRREVIEWER = {S. S. Shatz},
       URL = {http://www.numdam.org.libezp.lib.lsu.edu/item?id=PMIHES_1969__37__27_0},
}

@article {Fuselier-McCarthy,
    AUTHOR = {Fuselier, Jenny G. and McCarthy, Dermot},
     TITLE = {Hypergeometric type identities in the {$p$}-adic setting and
              modular forms},
   JOURNAL = {Proc. Amer. Math. Soc.},
  FJOURNAL = {Proceedings of the American Mathematical Society},
    VOLUME = {144},
      YEAR = {2016},
    NUMBER = {4},
     PAGES = {1493--1508},
      ISSN = {0002-9939},
   MRCLASS = {11F33 (11S80 33C20 33E50)},
  MRNUMBER = {3451227},
MRREVIEWER = {Jaban Meher},
       DOI = {10.1090/proc/12837},
       URL = {https://doi.org/10.1090/proc/12837},
}

@article {Win3X,
    AUTHOR = {Fuselier, Jenny and Long, Ling and Ramakrishna, Ravi and
              Swisher, Holly and Tu, Fang-Ting},
     TITLE = {Hypergeometric functions over finite fields},
   JOURNAL = {Mem. Amer. Math. Soc.},
  FJOURNAL = {Memoirs of the American Mathematical Society},
    VOLUME = {280},
      YEAR = {2022},
    NUMBER = {1382},
     PAGES = {},
      ISSN = {0065-9266},
      ISBN = {978-1-4704-5433-3; 978-1-4704-7282-5},
   MRCLASS = {11T23 (11-02 11F80 11S40 11T24 33)},
  MRNUMBER = {4493579},
       DOI = {10.1090/memo/1382},
       URL = {https://doi-org.libezp.lib.lsu.edu/10.1090/memo/1382},
}

@article {fu-li-wan-GKZ-padic,
    AUTHOR = {Fu, Lei and Li, Peigen and Wan, Daqing and Zhang, Hao},
     TITLE = {{$p$}-adic {GKZ} hypergeometric complex},
   JOURNAL = {Math. Ann.},
  FJOURNAL = {Mathematische Annalen},
    VOLUME = {387},
      YEAR = {2023},
    NUMBER = {3-4},
     PAGES = {1629--1689},
      ISSN = {0025-5831,1432-1807},
   MRCLASS = {14F30 (11T23 14G15 33C70)},
  MRNUMBER = {4657433},
MRREVIEWER = {Noriko\ Yui},
       DOI = {10.1007/s00208-022-02491-9},
       URL = {https://doi.org/10.1007/s00208-022-02491-9},
}

@article {Gross,
    AUTHOR = {Gross, Benedict H.},
     TITLE = {On the periods of abelian integrals and a formula of {C}howla
              and {S}elberg},
      NOTE = {With an appendix by David E. Rohrlich},
   JOURNAL = {Invent. Math.},
  FJOURNAL = {Inventiones Mathematicae},
    VOLUME = {45},
      YEAR = {1978},
    NUMBER = {2},
     PAGES = {193--211},
      ISSN = {0020-9910},
   MRCLASS = {14K22 (14K15 33A25)},
  MRNUMBER = {480542},
MRREVIEWER = {Neal Koblitz},
       DOI = {10.1007/BF01390273},
       URL = {https://doi-org.libezp.lib.lsu.edu/10.1007/BF01390273},
}

@incollection {Katz-Dwork,
    AUTHOR = {Katz, Nicholas M.},
     TITLE = {Travaux de {D}work},
 BOOKTITLE = {S\'{e}minaire {B}ourbaki, 24\`eme ann\'{e}e (1971/1972), {E}xp. {N}o.
              409},
     PAGES = {167--200. Lecture Notes in Math., Vol. 317},
      YEAR = {1973},
   MRCLASS = {14G13 (14G20)},
  MRNUMBER = {0498577},
  PUBLISHER = {Springer},
}

@article{HMM1,
      title={The {E}xplicit {H}ypergeometric-{M}odularity {M}ethod {I}}, 
      author={Michael Allen and Brian Grove and Ling Long and Fang-Ting Tu},   
JOURNAL = {Adv. Math.},
  FJOURNAL = {Advances in Mathematics},
    VOLUME = {487, Paper No. 110411},
      YEAR = {2025},
DOI = {10.1016/j.aim.2025.110411},
}

@article{HMM2,
      title={The {E}xplicit {H}ypergeometric-{M}odularity {M}ethod {II}}, 
      author={Michael Allen and Brian Grove and Ling Long and Fang-Ting Tu},
 JOURNAL ={Research in the Mathematical Sciences},
Volume={12, Paper No. 84},
     year={2025},
}

@article {LR,
    AUTHOR = {Long, Ling and Ramakrishna, Ravi},
     TITLE = {Some supercongruences occurring in truncated hypergeometric
              series},
   JOURNAL = {Adv. Math.},
  FJOURNAL = {Advances in Mathematics},
    VOLUME = {290},
      YEAR = {2016},
     PAGES = {773--808},
      ISSN = {0001-8708},
   MRCLASS = {33C20 (33E50)},
  MRNUMBER = {3451938},
MRREVIEWER = {Rupam Barman},
       DOI = {10.1016/j.aim.2015.11.043},
       URL = {https://doi.org/10.1016/j.aim.2015.11.043},
}

@incollection {Long18,
    AUTHOR = {Long, Ling},
     TITLE = {Some numeric hypergeometric supercongruences},
 BOOKTITLE = {Vertex operator algebras, number theory and related topics},
    SERIES = {Contemp. Math.},
    VOLUME = {753},
     PAGES = {139--156},
 PUBLISHER = {Amer. Math. Soc., Providence, RI},
      YEAR = {2020},
   MRCLASS = {11T24 (11A07 33C80)},
  MRNUMBER = {4139242},
       DOI = {10.1090/conm/753/15169},
       URL = {https://doi.org/10.1090/conm/753/15169},
}

@article {LTYZ,
    AUTHOR = {Long, Ling and Tu, Fang-Ting and Yui, Noriko and Zudilin,
              Wadim},
     TITLE = {Supercongruences for rigid hypergeometric {C}alabi-{Y}au
              threefolds},
   JOURNAL = {Adv. Math.},
  FJOURNAL = {Advances in Mathematics},
    VOLUME = {393},
      YEAR = {2021},
     PAGES = {Paper No. 108058, 49},
      ISSN = {0001-8708},
   MRCLASS = {11F33 (11T24 14G10 14J32 14J33 33C20)},
  MRNUMBER = {4330088},
MRREVIEWER = {Hidenori Katsurada},
       DOI = {10.1016/j.aim.2021.108058},
       URL = {https://doi-org.libezp.lib.lsu.edu/10.1016/j.aim.2021.108058},
}

@article {McCarthy-Papanikolas,
    AUTHOR = {McCarthy, Dermot and Papanikolas, Matthew A.},
     TITLE = {A finite field hypergeometric function associated to
              eigenvalues of a {S}iegel eigenform},
   JOURNAL = {Int. J. Number Theory},
  FJOURNAL = {International Journal of Number Theory},
    VOLUME = {11},
      YEAR = {2015},
    NUMBER = {8},
     PAGES = {2431--2450},
      ISSN = {1793-0421},
   MRCLASS = {11F46 (11F11 11G20 11T24 33E50)},
  MRNUMBER = {3420754},
MRREVIEWER = {Charles Helou},
       DOI = {10.1142/S1793042115501134},
       URL = {https://doi.org/10.1142/S1793042115501134},
}

@ARTICLE{LLT,
   author = {{Li}, Wen-Ching Winnie and {Long}, Ling and {Tu}, Fang-Ting},
    title = "{Computing special $L$-values of certain modular forms with complex multiplication}",
  journal = {SIGMA 14 (2018), 090},
archivePrefix = "arXiv",
   eprint = {1803.06072},
 primaryClass = "math.NT",
     year = 2018,
    month = aug,
   adsurl = {http://adsabs.harvard.edu/abs/2018arXiv180306072L}
}

@article {LLT2,
    AUTHOR = {Li, Wen-Ching Winnie and Long, Ling and Tu, Fang-Ting},
     TITLE = {A {W}hipple {$_7F_6$} formula revisited},
   JOURNAL = {Matematica},
  FJOURNAL = {La Matematica},
    VOLUME = {1},
      YEAR = {2022},
    NUMBER = {2},
     PAGES = {480--530},
   MRCLASS = {11F11 (11F67 11F80 33C20)},
  MRNUMBER = {4445932},
       DOI = {10.1007/s44007-021-00015-6},
       URL = {https://doi-org.libezp.lib.lsu.edu/10.1007/s44007-021-00015-6},
}

@book {Cohen-Stromberg,
    AUTHOR = {Cohen, Henri and Str\"{o}mberg, Fredrik},
     TITLE = {Modular forms},
    SERIES = {Graduate Studies in Mathematics},
    VOLUME = {179},
      NOTE = {A classical approach},
 PUBLISHER = {American Mathematical Society, Providence, RI},
      YEAR = {2017},
     PAGES = {xii+700},
      ISBN = {978-0-8218-4947-7},
   MRCLASS = {11-01 (11Fxx)},
  MRNUMBER = {3675870},
MRREVIEWER = {Sander Zwegers},
       DOI = {10.1090/gsm/179},
       URL = {https://doi-org.libezp.lib.lsu.edu/10.1090/gsm/179},
}

@article {McCarthy,
    AUTHOR = {McCarthy, Dermot},
     TITLE = {Transformations of well-poised hypergeometric functions over finite fields},
   JOURNAL = {Finite Fields Appl.},
  FJOURNAL = {Finite Fields and their Applications},
    VOLUME = {18},
      YEAR = {2012},
    NUMBER = {6},
     PAGES = {1133--1147},
      ISSN = {1071-5797},
   MRCLASS = {11T24 (11L99 33C20)},
  MRNUMBER = {3019189},
MRREVIEWER = {Jenny G. Fuselier},
       DOI = {10.1016/j.ffa.2012.08.007},
       URL = {http://dx.doi.org/10.1016/j.ffa.2012.08.007},
}

@incollection {RV-conj,
    AUTHOR = {Rodriguez-Villegas, Fernando},
     TITLE = {Hypergeometric families of {C}alabi-{Y}au manifolds},
 BOOKTITLE = {Calabi-{Y}au varieties and mirror symmetry ({T}oronto, {ON},
              2001)},
    SERIES = {Fields Inst. Commun.},
    VOLUME = {38},
     PAGES = {223--231},
 PUBLISHER = {Amer. Math. Soc., Providence, RI},
      YEAR = {2003},
   MRCLASS = {11G25 (11F33 14J32 33C20)},
  MRNUMBER = {2019156},
MRREVIEWER = {Scott Ahlgren},
}

@article {Yang04,
    AUTHOR = {Yang, Yifan},
     TITLE = {On differential equations satisfied by modular forms},
   JOURNAL = {Math. Z.},
  FJOURNAL = {Mathematische Zeitschrift},
    VOLUME = {246},
      YEAR = {2004},
    NUMBER = {1-2},
     PAGES = {1--19},
      ISSN = {0025-5874},
   MRCLASS = {11F11},
  MRNUMBER = {2031441},
MRREVIEWER = {F. Beukers},
       DOI = {10.1007/s00209-003-0573-4},
       URL = {https://doi-org.libezp.lib.lsu.edu/10.1007/s00209-003-0573-4},
}

@incollection {Zagier-top-diff,
    AUTHOR = {Zagier, Don},
     TITLE = {The arithmetic and topology of differential equations},
 BOOKTITLE = {European {C}ongress of {M}athematics},
     PAGES = {717--776},
 PUBLISHER = {Eur. Math. Soc., Z\"{u}rich},
      YEAR = {2018},
   MRCLASS = {11-02 (11Gxx 14H10 14J10 14J33)},
  MRNUMBER = {3890449},
}

@article {Zudilin-SIGMA,
    AUTHOR = {Zudilin, Wadim},
     TITLE = {A hypergeometric version of the modularity of rigid
              {C}alabi-{Y}au manifolds},
   JOURNAL = {SIGMA Symmetry Integrability Geom. Methods Appl.},
  FJOURNAL = {SIGMA. Symmetry, Integrability and Geometry. Methods and
              Applications},
    VOLUME = {14},
      YEAR = {2018},
     PAGES = {Paper No. 086, 16},
      ISSN = {1815-0659},
   MRCLASS = {11F33 (11T24 14G10 14J32 14J33 33C20)},
  MRNUMBER = {3844465},
MRREVIEWER = {Fang-Ting Tu},
       DOI = {10.3842/SIGMA.2018.086},
       URL = {https://doi.org/10.3842/SIGMA.2018.086},
}

@article {Shimura-special-zeta,
    AUTHOR = {Shimura, Goro},
     TITLE = {The special values of the zeta functions associated with cusp
              forms},
   JOURNAL = {Comm. Pure Appl. Math.},
  FJOURNAL = {Communications on Pure and Applied Mathematics},
    VOLUME = {29},
      YEAR = {1976},
    NUMBER = {6},
     PAGES = {783--804},
      ISSN = {0010-3640,1097-0312},
   MRCLASS = {10D15 (10H10)},
  MRNUMBER = {434962},
MRREVIEWER = {K.-B.\ Gundlach},
       DOI = {10.1002/cpa.3160290618},
       URL = {https://doi.org/10.1002/cpa.3160290618},
}

@incollection {deligne-l-values,
    AUTHOR = {Deligne, P.},
     TITLE = {Valeurs de fonctions {$L$}\ et p\'eriodes d'int\'egrales},
 BOOKTITLE = {Automorphic forms, representations and {$L$}-functions
              ({P}roc. {S}ympos. {P}ure {M}ath., {O}regon {S}tate {U}niv.,
              {C}orvallis, {O}re., 1977), {P}art 2},
    SERIES = {Proc. Sympos. Pure Math.},
    VOLUME = {XXXIII},
     PAGES = {313--346},
      NOTE = {With an appendix by N. Koblitz and A. Ogus},
 PUBLISHER = {Amer. Math. Soc., Providence, RI},
      YEAR = {1979},
      ISBN = {0-8218-1437-0},
   MRCLASS = {12A70 (10D15 10D24 10H10)},
  MRNUMBER = {546622},
MRREVIEWER = {James\ Milne},
}

@article {deligneraghuram,
    AUTHOR = {Deligne, Pierre and Raghuram, A.},
     TITLE = {Motives, periods, and functoriality},
   JOURNAL = {Tunis. J. Math.},
  FJOURNAL = {Tunisian Journal of Mathematics},
    VOLUME = {7},
      YEAR = {2025},
    NUMBER = {1},
     PAGES = {131--165},
      ISSN = {2576-7658,2576-7666},
   MRCLASS = {11F67 (11G09 22E55)},
  MRNUMBER = {4877275},
MRREVIEWER = {Lei\ Yang},
       DOI = {10.2140/tunis.2025.7.131},
       URL = {https://doi.org/10.2140/tunis.2025.7.131},
}

@misc{gugiatti2024hypergeometriclocalsystemsmathbbq,
      title={Hypergeometric local systems over $\mathbb{Q}$ with {H}odge vector $(1,1,1,1)$}, 
      author={Giulia Gugiatti and Fernando Rodriguez Villegas},
      year={2024, arXiv 2401.13529},
}

@misc{Barman_MaityEHMM,
      author={Rupam Barman and Sipra Maity},
      title= {Explicit hypergeometric modularity of certain weight two and four {H}ecke eigenforms}, 
      year={2026, arXiv 2604.02723},
}

@misc{EHMMcalc,
  shorthand    = {EHMM Calculator},
  author       = {Michael Allen and Brian Grove and Ling Long and Esme Rosen and Fang-Ting Tu},
  title        = {The {EHMM} {C}alculator},
  howpublished = {\url{https://sites.google.com/view/esme-rosen/ehmm-calculator}},
  year         = {2026},
  note         = {[Online; accessed 24 July 2026]},
}

@article {rosenK1,
    AUTHOR = {Rosen, Esme},
     TITLE = {{Modular Forms and Certain ${}_2F_1(1)$ Hypergeometric Series}},
    Journal = {Proc. of the Amer. Math. Soc.},
    FJournal = {Proceedings of the American Mathematical Society},
    Volume = {154},
    Year = {2026},
    Number = {7},
    Pages = {2803–2818},
    DOI ={https://doi.org/10.1090/proc/17616}
}

@article {rosen,
    AUTHOR = {Rosen, Esme},
     TITLE = {{$L$}-values of certain weight 3 Modular Forms and Transformations of  Hypergeometric Series},
     Volume = {13},
     Number ={57},
    Journal = {Res Math Sci},
    Fjournal={Research in the Mathematical Sciences},
Publisher = {},
      YEAR = {2026},
      DOI={https://doi.org/10.1007/s40687-026-00639-6},
    
}

@misc{rosenmixedI,
    author = {Rosen,Esme},
    title={Explicit Modularity of Reducible Rank 2 Hypergeometric Motives {I}},
    note = {In preparation}
}

@book{handbook,
  editor    = {Milton Abramowitz and Irene A. Stegun},
  title     = {Handbook of Mathematical Functions with Formulas, Graphs, and Mathematical Tables},
  series    = {National Bureau of Standards Applied Mathematics Series},
  volume    = {55},
  publisher = {U.S. Government Printing Office},
  address   = {Washington, D.C.},
  year      = {1964}
}

@incollection {Ramanujan-pi,
    AUTHOR = {Ramanujan, S.},
     TITLE = {Modular equations and approximations to {$\pi$} [{Q}uart. {J}.
              {M}ath. {\bf 45} (1914), 350--372]},
 BOOKTITLE = {Collected papers of {S}rinivasa {R}amanujan},
     PAGES = {23--39},
 PUBLISHER = {AMS Chelsea Publ., Providence, RI},
      YEAR = {2000},
      ISBN = {0-8218-2076-1},
   MRCLASS = {01A75},
  MRNUMBER = {2280849},
       URL = {},
}

@article{Borweincubic,
 author = {J. M. Borwein and P. B. Borwein},
 journal = {Transactions of the American Mathematical Society},
 number = {2},
 pages = {691--701},
 publisher = {American Mathematical Society},
 title = {A {C}ubic {C}ounterpart of {J}acobi's {I}dentity and the {AGM}},
 volume = {323},
 year = {1991}
}
